\documentclass[a4paper,reqno, 11pt]{amsart}

\usepackage[margin=3cm]{geometry}

\makeatletter
\let\old@setaddresses\@setaddresses
\def\@setaddresses{\bigskip\bgroup\parindent 0pt\let\scshape\relax\old@setaddresses\egroup}
\makeatother

\usepackage{amssymb} 
\usepackage{amsmath} 
\usepackage{amsthm} 
\usepackage[unicode=true]{hyperref}
\hypersetup{
    colorlinks,
    linkcolor={red!50!black},
    citecolor={blue!50!black},
    urlcolor={blue!80!black}
}

\usepackage{todonotes}
\usepackage{bbm}
\usepackage{paralist}
\usepackage[longnamesfirst,numbers,sort&compress]{natbib}
\usepackage[capitalise]{cleveref}
\usepackage[noend]{algorithmic}

\usepackage{pgfornament}

\newcommand{\pref}[1]{(P\ref{#1})}
\newcommand{\psref}[1]{(P\ref{#1}$'$)}
\newcommand{\nlref}[1]{(NL\ref{#1})}

\newenvironment{property}{\begin{enumerate}[(P1)]}{\end{enumerate}}
\newenvironment{pproperty}{\begin{enumerate}[(P1$'$)]}{\end{enumerate}}

\newtheorem{theorem}{Theorem}

\newtheorem{lemma}[theorem]{Lemma}

\newtheorem{observation}[theorem]{Observation}

\newtheorem{claim}[theorem]{Claim}

\newcommand{\comparable}{\mathbin{\diamondsuit}}

\newcommand{\F}{\mathcal{F}}

\newcommand{\Q}{\mathcal{Q}}

\DeclareMathOperator{\h}{height}
\DeclareMathOperator{\R}{\mathbb{R}}
\DeclareMathOperator{\N}{\mathbb{N}}

\let\le\leqslant
\let\ge\geqslant
\let\leq\leqslant
\let\geq\geqslant

\usepackage{subfig}

\graphicspath{{figures/}}

\begin{document}

\title{Sparse universal graphs for planarity}

\author[L.~Esperet]{Louis Esperet}
\address[L.~Esperet]{Laboratoire G-SCOP (CNRS, Univ.\ Grenoble Alpes), Grenoble, France}
\email{louis.esperet@grenoble-inp.fr}

\author[G.~Joret]{Gwena\"el Joret}
\address[G.~Joret]{D\'epartement d'Informatique,
Universit\'e libre de Bruxelles,
Brussels, Belgium}
\email{gwenael.joret@ulb.be}

\author[P.~Morin]{Pat Morin}
\address[P.~Morin]{School of Computer Science, Carleton University,
  Canada}
\email{morin@scs.carleton.ca}

\thanks{L.\ Esperet is partially supported by the French ANR Projects ANR-16-CE40-0009-01 (GATO) and ANR-18-CE40-0032 (GrR). G.\ Joret is supported by an ARC grant from the Wallonia-Brussels Federation of Belgium and a CDR grant from the National Fund for Scientific Research (FNRS). P.\ Morin is partially supported by NSERC}

\date{\today}
\sloppy

\begin{abstract}
We show that for every integer $n\geq 1$ there exists a graph $G_n$
with $(1+o(1))n$ vertices and $n^{1 + o(1)}$ edges such that every
$n$-vertex planar graph is isomorphic to a subgraph of $G_n$. The best
previous bound on the number of edges was $O(n^{3/2})$, proved by
Babai, Chung, Erd\H{o}s, Graham, and Spencer in 1982. We then show
that for every integer $n\geq 1$ there is a graph $U_n$
with $n^{1 + o(1)}$ vertices and edges that contains induced copies
of  every
$n$-vertex planar graph. This significantly reduces the number of
edges in a recent construction of the authors with Dujmovi\'c,
Gavoille, and Micek.
\end{abstract}


\maketitle

\section{Introduction}\label{sec:intro}

Given a family $\F$ of graphs, a graph $G$ is {\em universal} for $\F$ if every graph in $\F$ is isomorphic to a (not necessarily induced) subgraph of $G$.
The topic of this paper is the following question: What is the minimum
number of edges in a universal graph for the family of $n$-vertex planar graphs?
Besides being a natural question, we note that finding sparse
universal graphs is also motivated by
applications in VLSI design~\cite{bhatt.chung.ea,Val81} and simulation
of parallel computer architecture~\cite{bhatt1988optimal,bhatt1986optimal}.

A moment's thought shows that $\Omega(n \log n)$ edges are needed: for
$t= 1, \dots, n$, consider the forest consisting of $t$ copies of the star
$K_{1, \lfloor n/t\rfloor  - 1}$. A universal graph for the class of
$n$-vertex planar graphs must contain all these forests as subgraphs,
and so it must have a
degree sequence which, once sorted in non-increasing order, dominates
the sequence
\[(n-1, \lfloor \tfrac{n}2\rfloor -1, \lfloor \tfrac{n}3\rfloor -1, \lfloor \tfrac{n}4\rfloor -1,\ldots),
\]
hence the lower bound.
As far as we are aware, no better lower bound is known for $n$-vertex planar graphs.

For $n$-vertex trees, a matching upper bound of $O(n \log n)$ on the
number of edges in the universal graph is
known~\cite{chung.graham}. For $n$-vertex planar graphs of bounded
maximum degree, Capalbo constructed a universal graph with $O(n)$ edges~\cite{Capalbo02}.
However, for general $n$-vertex planar graphs only a $O(n^{3/2})$
bound is known, proved by Babai, Chung, Erd\H{o}s, Graham, and
Spencer~\cite{babai.chung.ea} in 1982 using the existence of
separators of size $O(\sqrt{n})$.

\smallskip

In this paper we show that universal graphs with a near-linear number of edges can be constructed:

\begin{theorem}\label{thm:planar}
  The family of $n$-vertex planar graphs has a universal graph with
  $(1+o(1))n$ vertices and at
  most $n\cdot 2^{O(\sqrt{\log n \cdot \log \log n})}$ edges.
\end{theorem}



As the original construction of Babai et al.~\cite{babai.chung.ea}
only uses the existence of
separators of size $O(\sqrt{n})$, it was later shown to apply to more general classes
than planar graphs, for instance to any proper minor-closed
class~\cite{chung1990-separator}. Our result also holds in greater
generality, but not quite as general as the construction of  Babai et
al.~\cite{babai.chung.ea}, as we now explain.

\medskip

The \emph{strong product} $A\boxtimes B$ of two graphs $A$ and $B$ is the graph whose vertex set is the Cartesian product $V(A\boxtimes B):=V(A)\times V(B)$ and in which two distinct vertices $(x_1,y_1)$ and $(x_2,y_2)$ are adjacent if and only if:
\begin{enumerate}
  \item  $x_1x_2 \in E(A)$ and $y_1y_2 \in E(B)$; or
  \item $x_1=x_2$ and $y_1y_2\in E(B)$; or
  \item $x_1x_2 \in E(A)$ and $y_1=y_2$.
  \end{enumerate}

\smallskip

We may now state the main result of this paper.

\begin{theorem}\label{thm:main}
Fix a positive integer $t$ and let $\Q_t$ denote the family of all graphs of the form $H\boxtimes P$ where $H$ is a graph of treewidth $t$ and $P$ is a path, together with all their subgraphs.
Then the family of $n$-vertex graphs in $\Q_t$ has a universal graph
with $(1+o(1))n$ vertices and at most $t^2\cdot n\cdot 2^{O(\sqrt{\log n \cdot \log \log
    n})} $ edges.
\end{theorem}

It was proved by Dujmovi\'c, Joret, Micek, Morin, Ueckerdt
and  Wood \cite{dujmovic.joret.ea:planarJACM} that every planar graph is
a subgraph of the strong product of a graph of treewidth at most 8
and a path (see also the recent improvement by Ueckerdt, Wood, and Yi~\cite{UWY22}).

\begin{theorem}[\cite{dujmovic.joret.ea:planarJACM}]\label{product-structure}
  The class of planar graphs is a subset of $\Q_8$.
\end{theorem}

Moreover, Bose, Morin, and Odak~\cite{bose_et_al:LIPIcs.SWAT.2022.19} gave a linear-time algorithm that given a planar graph $G$, finds a graph $H$ of treewidth at most 8 and an embedding of $G$ in the strong product of $H$ with a path.

\medskip

Note that \cref{product-structure} and \cref{thm:main} directly imply \cref{thm:planar}.
It was proved that \cref{product-structure} can be generalized (replacing $8$ with a
larger constant) to bounded genus graphs, and more generally to
apex-minor free graphs~\cite{dujmovic.joret.ea:planarJACM}, as well as
to $k$-planar graphs and related classes of
graphs~\cite{dujmovic.morin.ea:structure}. Thus it follows that
families of
$n$-vertex graphs in these more general classes also admit universal
graphs with $n^{1+o(1)}$ edges.

\medskip

\noindent\textbf{Induced-universal graphs.} A related problem is to
find an \emph{induced-universal graph} for a family $\F$, which is a
graph that contains all the graphs of $\F$ as \emph{induced} subgraphs. In
this context the problem is usually to minimize the number of vertices of the induced-universal graph~\cite{kannan.naor.ea:implicit}. Recently, Dujmovi\'c, Esperet, Joret, Gavoille, Micek and Morin used \cref{product-structure} to construct an induced-universal
graph with $n^{1+o(1)}$ vertices for the class of $n$-vertex planar
graphs~\cite{AdjacencyLabellingPlanarFOCS, AdjacencyLabellingPlanarJACM}. Since an induced-universal
graph for a class $\F$ is also universal for $\F$, their graph is
universal for the class of $n$-vertex planar graphs. However, while that graph has a near-linear number of vertices, it is quite dense, it has order of $n^2$ edges. Thus, it is not directly useful in the context of minimizing the number of edges.

Nevertheless, in this paper we reuse key ideas and techniques introduced in~\cite{AdjacencyLabellingPlanarJACM}.
Very informally, a central idea in~\cite{AdjacencyLabellingPlanarJACM} is the notion of {\em bulk tree sequences}, which is used to efficiently `encode' the rows from the product structure using almost perfectly balanced binary search trees, in such a way that the trees undergo minimal changes when moving from one row to the next one. (These tree sequences are described in the next section.)

Given that, for $n$-vertex planar graphs, there exist (1) a universal graph with a near-linear number of edges, and (2) an induced-universal graph with a near-linear number of vertices, it is natural to wonder if these two properties can be achieved simultaneously.  In the second part of this paper, we show that this can be done.

\begin{theorem}\label{thm:planar_induced}
  The family of $n$-vertex planar graphs has an induced-universal graph with at
  most $n\cdot 2^{O(\sqrt{\log n \cdot \log \log n})}$ edges and vertices.
\end{theorem}

In the same way that \cref{thm:planar} is a special case of \cref{thm:main}, \cref{thm:planar_induced} is obtained as a special case of \cref{thm:main_induced}:

\begin{theorem}\label{thm:main_induced}
Fix a positive integer $t$.
Then the family of $n$-vertex graphs in $\Q_t$ has an induced-universal graph
with at most $n\cdot 2^{O(\sqrt{\log n \cdot \log \log n})} \cdot (\log n)^{O(t^2)}$ edges and vertices.
\end{theorem}

The construction in \cref{thm:main_induced} is based on a non-trivial
modification of the construction of induced-universal graphs
in~\cite{AdjacencyLabellingPlanarJACM} and reuses ideas from the
construction of universal graphs in the first part of the current
paper. It is significantly more complicated than the construction used
for \cref{thm:planar}, is more tightly coupled with the labelling
scheme in \cite{AdjacencyLabellingPlanarJACM}, and the end result has
a greater dependence on $t$ (a $t^2$ factor in \cref{thm:main} is
replaced by a $(\log n)^{O(t^2)}$ factor in
\cref{thm:main_induced}). Moreover, the classical techniques that
allow us to reduce the number of vertices from $n^{1+o(1)}$ to
$(1+o(1))n$ in \cref{thm:main} do not apply to induced-universal
graphs, so decreasing further the number of vertices in
\cref{thm:main_induced} seems to require completely new ideas.

\medskip

\noindent\textbf{Paper organization.}
The first part of the paper consists of Sections~\ref{sec:prel} and~\ref{sec:universal}, and is devoted to proving \cref{thm:main}.
In the second part of the paper, Section~\ref{sec:universal}, we start by recalling the construction of induced-universal graphs from~\cite{AdjacencyLabellingPlanarJACM}.
Then, we explain why these graphs are too dense, and describe how to modify the construction to achieve a near-linear number of edges.

\section{Preliminaries}\label{sec:prel}

\subsection{Graph products}\label{sec:gp}

Given two graphs $G_1,G_2$, and $v_1\in V(G_1)$, the set
$\{(v_1,v_2)\mid v_2\in V(G_2)\}$ is called a \emph{column} of
$G_1\boxtimes G_2$. Similarly, for $v_2\in V(G_2)$, the set $\{(v_1,v_2)\mid v_1\in V(G_1)\}$ is called a \emph{row} of
$G_1\boxtimes G_2$.

\begin{lemma}\label{obs:stproduct}
Let $G_1$ and $G_2$ be two graphs, and let $H$ be an $n$-vertex
subgraph of $G_1\boxtimes G_2$. Then $G_1$ and $G_2$ contain induced
subgraphs $G_1'$ and $G_2'$ with at most $n$ vertices such that
$H$ is a subgraph of $G_1'\boxtimes G_2'$, and each row and column of $G_1'\boxtimes G_2'$ contains at least one vertex of $H$.
\end{lemma}

\begin{proof}
If there is a vertex $x\in V(G_1)$ such that no vertex of $G_1\boxtimes G_2$ of the form $(x,y)$ is included in $H$, then $H$ is a subgraph of
$(G_1-x)\boxtimes G_2$. So, by considering induced subgraphs $G_1'$
and $G_2'$ of $G_1$ and $G_2$ if necessary, we may assume that each column (and by symmetry each row) of $G_1'\boxtimes G_2'$ contains a vertex of
$H$. It follows that $G_1'$ and $G_2'$ contain at most $n$ vertices.
\end{proof}

We deduce the following result, which will be useful in the proof of
our main result.

\begin{lemma}\label{obs:cliqueproduct}
Let $n$ be an integer, let $H_1$ be a graph with at least $n$
vertices, and let $G_1$ be a graph that is
universal for the family of $n$-vertex subgraphs of $H_1$. Then for
any graph $H_2$, the graph $G_1\boxtimes H_2$ is universal for the
family of $n$-vertex subgraphs of $H_1\boxtimes H_2$.
\end{lemma}

\begin{proof}
Let $G$ be an $n$-vertex subgraph of $H_1\boxtimes H_2$. By
Lemma \ref{obs:stproduct} we can assume that there is a subgraph
$H_1'$ of $H_1$  with at most $n$ vertices,
such that $G$ is a subgraph of $H_1'\boxtimes H_2$. By adding
vertices of $H_1$ to $H_1'$ if necessary, we can assume that $H_1'$
contains precisely $n$ vertices, and is thus a subgraph of  $G_1$. It
follows that $G$ is a subgraph of $G_1\boxtimes H_2$, as desired.
\end{proof}

\subsection{Binary Search Trees}\label{sec:bst}

A \emph{binary tree} $T$ is a rooted tree in which each node except
the root is either the \emph{left} or \emph{right} child of its parent
and each node has at most one left and at most one right child.  For
any node $x$ in $T$, $P_T(x)$ denotes the path from the root of $T$ to
$x$.  The \emph{length} of a path $P$ is the number of edges in $P$,
i.e., $|P|-1$.  The \emph{depth}, $d_T(x)$ of $x$ is the length of
$P_T(x)$.  The \emph{height} of $T$ is $\h(T):=\max_{x\in V(T)}
d_T(x)$.
A node $x$ in $T$ is a \emph{$T$-ancestor} of a node $y$ in $T$ if $x\in V(P_T(y))$. If $x$ is a $T$-ancestor of $y$ then $y$ is a \emph{$T$-descendant} of $x$. A $T$-ancestor $x$ of $y$ is a \emph{strict $T$-ancestor} if $x\neq y$.  We use $\prec_T$ to denote the strict $T$-ancestor relation and $\preceq_T$ to denote the $T$-ancestor relation.
Let $P_T(x_r)=x_0,\dots,x_{r}$ be a path from the root $x_0$ of $T$ to some node $x_r$ (possibly $r=0$).  Then the \emph{signature} of $x_r$ in $T$, denoted $\sigma_T(x_r)$ is a binary string $b_1,\dots,b_r$ where $b_i=0$ if and only if $x_{i}$ is the left child of $x_{i-1}$.
Note that the signature of the root of $T$ is the empty string.

A \emph{binary search tree} $T$ is a binary tree whose node set $V(T)$
consists of distinct real numbers and that has the \emph{binary search
  tree property}:  For each node $x$ in $T$, $z<x$ for each node $z$
in $x$'s left subtree and $z>x$ for each node $z$ in $x$'s right
subtree.

\medskip

Let $\log x:=\log_2 x$ denote the binary logarithm of $x$.
We will use the following standard facts about binary search trees, which were also used in~\cite{AdjacencyLabellingPlanarJACM}.

\begin{lemma}[Lemma 5 in~\cite{AdjacencyLabellingPlanarJACM}]\label{lem:biased-bst}
  For any finite $S\subset \R$ and any function $w:S\to\R^+$, there exists a binary search tree $T$ with $V(T)=S$ such that, for each $y\in S$, $d_T(y)\le\log(W/w(y))$, where $W:=\sum_{y\in S} w(y)$.
\end{lemma}

\begin{observation}[Observation 6 in~\cite{AdjacencyLabellingPlanarJACM}]\label{obs:successor-encoding}
  Let $T$ be a binary search tree and let $x,y$ be nodes in $T$ such that $x<y$ and there is no node $z$ in $T$ such that $x<z<y$, i.e., $x$ and $y$ are consecutive in the sorted order of $V(T)$.  Then
  \begin{enumerate}
    \item (if $x$ has no right child) $\sigma_T(y)$ is obtained from $\sigma_T(x)$ by removing all trailing 1's and the last 0; or
    \item (if $x$ has a right child) $\sigma_T(y)$ is obtained from $\sigma_T(x)$ by appending a 1 followed by $d_T(y)-d_T(x)-1$ 0's.
  \end{enumerate}

Therefore, for each $\sigma \in \{0,1\}^*$ and integer $h$ such that
$|\sigma| \leq h$, there exists a set $L(\sigma, h)$ of bitstrings in $\{0,1\}^*$, each of length at most $h$, with $|L(\sigma, h)| \leq h+1$ such that for every binary search tree $T$ of height at most $h$ and for every two consecutive nodes $x, y$ in the sorted order of $V(T)$, we have $\sigma_T(y) \in L(\sigma_T(x),h)$.
\end{observation}

\subsection{Universal graphs for interval graphs}\label{sec:interval}

An {\em interval graph} is a graph $G$ that admits an {\em interval representation}, defined  as a collection $(I_v)_{v\in V(G)}$ of closed intervals of the real line such that, for distinct vertices $v,w\in V(G)$, $vw\in E(G)$ if and only if $I_v \cap I_w \neq \emptyset$.

\begin{lemma}\label{lem:intervalsep}
Let $G$ be an $n$-vertex interval graph with clique number at most $\omega$. Then
$V(G)$ can be partitioned into three sets $X_1,X_2,Z$ such that $|Z|\le
\omega$,
$|X_i|\le  \tfrac12 n$
for $i\in \{1,2\}$, and there are no
edges between $X_1$ and $X_2$.
\end{lemma}

\begin{proof}
Consider an interval representation $(I_v)_{v\in V(G)}$ of $G$, where
$I_v=[a_v,b_v]$ for any $v\in V(G)$, and such that at most one
interval $I_v$ starts at each point (it is well known that such a
representation exists). Order the vertices of $G$ as
$v_1,\ldots,v_n$ such that for any $1\le i\le j\le n$, $a_{v_i}\le
a_{v_j}$. For each $1\le i \le n$, let $Z_i$ be the set of vertices
$v$ of $G$ such that $I_v$ contains $a_{v_i}$. Since $G$ has clique
size at most $\omega$, each set $Z_i$ contains at most $\omega$
vertices (and at least one vertex, namely $v_i$). Moreover, the vertex set of each $G-Z_i$ can be partitioned
into two sets $A_i$ (the vertices $v$ such that $b_v<a_{v_i}$) and
$B_i$ (the vertices $v$ such that $a_v>a_{v_i}$) with no edges
between them. Recall that at most one interval starts at
each $a_{v_i}$, so $|B_i|=n-i$ for any $1\le i \le n$. So there is
$1\le i \le n$ such that $n/2-1\le |B_i|\le n/2$. It follows that
$|A_i|\le n/2+1-|Z_i|\le n/2$, as desired.
\end{proof}

\begin{figure}[htb]
 \centering
 \includegraphics[scale=0.8]{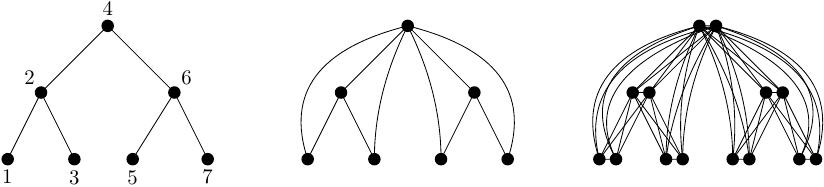}
 \caption{From left to right: $B_2$, $C_2$, and $C_2\boxtimes K_2$.}
 \label{fig:closure}
\end{figure}

For any integer $d\ge 0$, let $B_d$ be the unique binary search tree
with $V(B_d)=\{1,\ldots,2^{d+1}-1\}$ and having height $d$. The
\emph{closure} $C_d$ of $B_d$ is the graph with vertex set
$V(C_d):=V(B_d)$ and edge set $E(C_d):=\{vw:v\prec_{B_d} w\}$ (see
Figure~\ref{fig:closure}). The universal graph for the family of
$n$-vertex planar graphs of Babai et al.~\cite{babai.chung.ea}, with
$O(n^{3/2})$ edges, is precisely $C_{\lceil \log n \rceil} \boxtimes
K_t$, with $t=O(\sqrt{n})$. Using the same idea, we now describe
a universal graph for the family of $n$-vertex interval graphs of bounded clique number.

\begin{lemma}\label{lem:intervaluniversal}
For every positive integers $n\ge 1$ and $\omega\ge 1$, the graph $C_{\lceil \log n \rceil} \boxtimes K_\omega$ is universal for the class of $n$-vertex interval graphs with clique number at most $\omega$.
\end{lemma}

\begin{proof}
We prove the result by induction on $n$. If $n=1$, then the result
clearly holds, so we can assume that $n\ge 2$. Consider an
$n$-vertex interval graph $G$ with clique number at most $\omega$. By
Lemma \ref{lem:intervalsep}, the vertex set of $G$ has a partition into
three sets  $X_1,X_2,Z$ such that $|Z|\le
\omega$,
$|X_i|\le  \tfrac12 n$
for $i\in \{1,2\}$, and there are no
edges between $X_1$ and $X_2$. By the induction hypothesis, $G[X_1]$
is a subgraph of $C_{\lceil \log (n/2) \rceil} \boxtimes K_\omega=C_{\lceil \log n \rceil-1} \boxtimes K_\omega$ and similarly $G[X_2]$ is a subgraph
of  $C_{\lceil \log n \rceil-1} \boxtimes K_\omega$. Note that for $n\ge 2$, $C_{\lceil \log n \rceil} \boxtimes K_\omega$ can be obtained from two disjoint copies of $C_{\lceil \log n \rceil-1} \boxtimes K_\omega$ by adding $\omega$ universal
vertices. Using that $|Z|\le
\omega$, this implies that $G$ is a subgraph of $C_{\lceil \log n \rceil} \boxtimes K_\omega$, as desired.
\end{proof}

Note that the proof of Lemma \ref{lem:intervaluniversal} is constructive:
given any interval representation of an $n$-vertex
interval graph $G$ with clique number at most $\omega$, it gives an efficient deterministic  algorithm to
find a copy of $G$ in $C_{\lceil \log n
  \rceil} \boxtimes K_\omega$.

\medskip

For a node $v\in C_d$, define the \emph{interval} $I_{C_d}(v):=\{w\in
V(B_d):v\preceq_{B_d} w\}$. We observe that
any two intervals are either nested or disjoint.

\begin{observation}\label{proper_containment}
    For any two nodes $v,w$ of $C_d$,  $I_{C_d}(v)\supseteq I_{C_d}(w)$, $I_{C_d}(v) \subseteq I_{C_d}(w)$ or $I_{C_d}(v)\cap I_{C_d}(w)=\emptyset$ and, in the first two cases, $vw\in E(C_d)$.
\end{observation}

Let $G$ be an induced subgraph of $C_d$ and let $T$ be a binary search tree with $V(G)\subseteq V(T)\subseteq V(C_d)$.
(Let us remark that, while the node set of $T$ is a subset of that of $B_d$, the structure of $T$ could potentially be very different from that of $B_d$.)
For each $v\in V(G)$, let $x_T(v)$ denote the node $x\in V(T)$ of minimum $T$-depth such that $x\in I_{C_d}(v)$.  Note that, for each $v\in V(G)$,  $x_T(v)$ is well-defined since $v\in V(T)$ and $v\in I_{C_d}(v)$.

For two strings $x$ and $y$ we use $x\preceq y$ to denote that $x$ is a prefix of $y$, and $x\prec y$ to denote that $x\preceq y$ and $|x|<|y|$.  We use $x\comparable y$ to denote
that $x\preceq y$ or $y\preceq x$ (note that the relation
$\comparable$ is reflexive and symmetric but not transitive).

\begin{lemma}\label{consistent_ancestor}
  Let $G$ be an induced subgraph of $C_d$ and let $T$ be a binary search tree with $V(G)\subseteq V(T)\subseteq V(C_d)$.
Let $vw \in E(G)$.  Then $x_T(v) \preceq_T x_T(w)$ or $x_T(w) \preceq_T x_T(v)$, and hence $\sigma_T(x_T(v)) \comparable \sigma_T(x_T(w))$.
\end{lemma}
\begin{proof}
Note that $vw \in E(G)$ implies that $I_{C_d}(v)\subseteq I_{C_d}(w)$ or $I_{C_d}(v)\supseteq I_{C_d}(w)$, say without loss of generality $I_{C_d}(v)\subseteq I_{C_d}(w)$.
Suppose that neither $x_T(v) \preceq_T x_T(w)$ nor $x_T(w) \preceq_T x_T(v)$ holds.
Then there exists a common $T$-ancestor $z \in V(T)$ of $x_T(v)$ and $x_T(w)$ with $z\neq x_T(v), x_T(w)$.
Since $T$ is a binary search tree, it follows that $x_T(v) < z < x_T(w)$ or $x_T(w) < z < x_T(v)$.
Since $x_T(v) \in I_{C_d}(v)\subseteq I_{C_d}(w)$ and $x_T(w) \in I_{C_d}(w)$, we also have $z \in I_{C_d}(w)$, by definition of $I_{C_d}(w)$.
Hence, $z \in I_{C_d}(w)$ and $z$ is a strict $T$-ancestor of $x_T(w)$, which contradicts the choice of $x_T(w)$.
\end{proof}

The following lemma from~\cite{AdjacencyLabellingPlanarJACM} is a key
tool in our proof.

\begin{lemma}[Lemmas 8, 25 and 35 in~\cite{AdjacencyLabellingPlanarJACM}]\label{lem:tree_sequence}
  Let $n$ be a positive integer and define $k=\max(5,\lceil\sqrt{\log
        n / \log\log n}\rceil)$. Then there exists a function $B:(\{0,1\}^*)^2\to\{0,1\}^*$ such that, for
  any finite sets $S_1,\dots,S_h\subset\R$ with $\sum_{y=1}^h |S_y|=n$, there exist binary search
  trees $T_1,\dots,T_h$ such that
  \begin{compactenum}[(1)]
    \item \label{pretend_appear} for each $y\in\{1,\dots,h-1\}$, $V(T_y)\supseteq S_{y}\cup S_{y+1}$, and $V(T_h)\supseteq S_h$;
    \item \label{bound_sum} $\sum_{y=1}^h |V(T_y)|\le 4\sum_{y=1}^h |S_y|=4n$; \hfill
      {\emph{\cite[Lemma 8]{AdjacencyLabellingPlanarJACM}}}
      \item \label{bound_height} for each $y\in\{1,\dots,h\}$, $\h(T_y)\le \log|V(T_y)| +
        O\left(k+k^{-1} \log|V(T_y)|\right).$
\hfill {\emph{\cite[Lemma 25]{AdjacencyLabellingPlanarJACM}}}
        \item  \label{bound_nu}
                  for each $y\in\{1,\dots,h-1\}$, and each $w\in
                  S_{y+1}$, there exists $\nu_y(w)\in\{0,1\}^*$ with
                  $|\nu_y(w)| = O(k\log (\h(T_y)))$ such that
                  $B(\sigma_{T_y}(x_{T_y}(w)), \nu_y(w)) =
                  \sigma_{T_{y+1}}(x_{T_{y+1}}(w))$. \hfill
                  {\emph{\cite[Lemma
                      35]{AdjacencyLabellingPlanarJACM}}} \footnote{We
                    note that our presentation here differs slightly
                    from \cite[Lemma
                      35]{AdjacencyLabellingPlanarJACM}, see Footnote
                  \ref{foot:P6} for more details.}
  \end{compactenum}
\end{lemma}

The sequence $T_1,\dots,T_h$ obtained in the lemma
is called a \emph{bulk tree sequence}
in~\cite{AdjacencyLabellingPlanarJACM}, and plays a fundamental role
in~\cite{AdjacencyLabellingPlanarJACM} and the present paper. 

\begin{observation}\label{obs:lambda}
There exists a function $\lambda:\N \to \N$ with $\lambda(n) \in O(\sqrt{\log n \log\log n})$ such that
\begin{itemize}
    \item $\h(T_y)\le \log|V(T_y)| + \lambda(n)$ always holds in
      property~\eqref{bound_height} of Lemma \ref{lem:tree_sequence}, and
    \item $|\nu_y(z)| \leq \lambda(n)$ always holds in
      property~\eqref{bound_nu} of Lemma \ref{lem:tree_sequence}.
\end{itemize}
\end{observation}
\begin{proof}
This follows from the bounds $\h(T_y)\le \log|V(T_y)| +
O\left(k+k^{-1} \log|V(T_y)|\right) $  in
property~\eqref{bound_height} of Lemma \ref{lem:tree_sequence} and
$|\nu_y(z)| = O(k\log (\h(T_y)))$ in property~\eqref{bound_nu} of
Lemma \ref{lem:tree_sequence}, combined with properties \eqref{bound_sum} and \eqref{bound_height} of that lemma.
\end{proof}

It is important to note that the function $L$ of Observation \ref{obs:successor-encoding} and the
function $B$ of Lemma \ref{lem:tree_sequence} are \emph{explicit}, in the
sense that
\cite{AdjacencyLabellingPlanarJACM} provides simple deterministic
algorithms for producing the output of the functions (note that this
is clear for Observation \ref{obs:successor-encoding} by considering
(1) and (2) in the statement of the observation).

\subsection{Treewidth and pathwidth}\label{sec:tw}

A \emph{tree-decomposition} of a graph $G$ is a tree $T$ along with a
collection of subsets $(X_t)_{t\in V(T)}$ of vertices of $G$ (called the
\emph{bags} of the decomposition) such that for every edge $uv\in E(G)$,
there is a node $t\in V(T)$ such that $u,v\in X_t$, and for every
vertex $u \in V(G)$, the nodes $t$ of $T$ such that $u\in X_t$ form a
(non-empty) subtree of $T$. The tree-decomposition is called a
\emph{path-decomposition} if the tree $T$ is a path. The \emph{width} of a tree-decomposition is the
maximum size of a bag, minus 1. The \emph{treewidth} of a graph $G$ is
the minimum width of a tree-decomposition of $G$, and the \emph{pathwidth}
of a graph $G$ is the minimum width of a path-decomposition of $G$. Note that the treewidth of a graph $G$ is
at most the pathwidth of $G$. We will use the following partial
converse.

\begin{lemma}[\cite{scheffler:optimal}]\label{lem:twpw}
Every $n$-vertex graph of
treewidth at most $t$ has pathwidth at most $(t+1) \lfloor \log_3 (2n+1)  + 1 \rfloor - 1$.
\end{lemma}

Observe that an equivalent definition of pathwidth, which will be used in the proofs, is the following: A graph $G$ has pathwidth at most $k$ if and only if $G$ is a spanning\footnote{A subgraph $G$ of a graph $H$ is {\em spanning} if $V(G)=V(H)$.} subgraph of an interval graph with clique number at most $k+1$.

\section{Universal graphs}\label{sec:universal}

In this section we establish the following technical theorem.

\begin{theorem}\label{thm:main_technical}
For every positive integer $n$, the family of $n$-vertex
induced subgraphs of $C_{\lceil \log n \rceil}\boxtimes P_n$ has a
universal graph $G_n$ with
\[|V(G_n)|\le n\cdot 2^{O(\sqrt{\log n \cdot \log \log n})} \text{ and
  }|E(G_n)|\le n\cdot 2^{O(\sqrt{\log n \cdot \log \log n})}.
  \]
\end{theorem}

Before proving \cref{thm:main_technical}, let us explain why it
implies our main theorem, \cref{thm:main}. The proof proceeds in two
steps: we first show that for $\omega\approx t \log n$, $G_n \boxtimes
K_\omega$ is a universal graph for the $n$-vertex graphs of
$\Q_t$. This graph has the desired number of edges, but a fairly large
number of vertices. The second step of the proof consists in reducing
the number of vertices to $(1+o(1))n$.

  We say that a subset $X$ of vertices of a graph $G$ is
  \emph{saturated} by a matching $M$ of $G$ if every vertex of $X$
  is contained in some edge of $M$. We
will need the following result proved (in a slightly different form) in \cite{ACKRRS}. As the proof there is only alluded to, we give the complete details
in the appendix.


\begin{lemma}\label{lem:alon}
For any sufficiently large integer $n$, any integer $k\ge 1$, any real
number $0<\epsilon\le 1$, and any integer $N_0\ge k(1+\epsilon) n$,
there is a bipartite graph $H$ with
bipartition $(V,U)$ such that the following holds.
\begin{itemize}
\item $|V|=N$, and $|U|=N/k$, where $N$ is
  divisible by $k$ and $N_0\le N\le N_0+k$,
  \item each vertex of $V$
    has degree $O(\tfrac1\epsilon \log k)$, and
    \item each $n$-vertex subset of $V$ is
      saturated by a matching of $H$.
    \end{itemize}
  \end{lemma}

We are now ready to prove \cref{thm:main}.

\begin{proof}[Proof of \cref{thm:main} assuming
  \cref{thm:main_technical}]
Let $G_n$ be the universal graph for the family of $n$-vertex
subgraphs of $C_{\lceil \log n \rceil}\boxtimes P_n$ given by
\cref{thm:main_technical}.
Let $t$ be an integer and let $\omega=(t+1) \lfloor \log_3 (2n+1)  + 1
\rfloor$. By Lemma \ref{obs:cliqueproduct}, $G_n'=G_n\boxtimes
K_\omega$ is universal for the class of $n$-vertex subgraphs of $C_{\lceil \log n \rceil} \boxtimes P_n \boxtimes K_\omega$. Note that $G_n'$ has precisely $\omega |V(G_n)|=\omega \cdot n\cdot 2^{O(\sqrt{\log
    n \cdot \log \log n})} $ vertices and
\[|E(G_n)|\cdot\omega^2+|V(G_n)|\cdot{\omega \choose
  2}\le  \omega^2\cdot n\cdot 2^{O(\sqrt{\log n \cdot \log \log
    n})}\le t^2\cdot n\cdot 2^{O(\sqrt{\log n \cdot \log \log
    n})}
\]
edges. We will see shortly how to reduce the number of vertices from
$\omega n\cdot 2^{O(\sqrt{\log n \cdot \log \log n})}$ to $(1+o(1))n$, but
for now we prove that $G_n'$ is universal for the $n$-vertex graphs of
$\Q_t$. For this it suffices to show that any $n$-vertex graph $G\in
\Q_t$ is a subgraph of $C_{\lceil \log n \rceil} \boxtimes P_n
\boxtimes K_\omega$.

\smallskip

We consider an $n$-vertex graph $G\in \Q_t$. By the definition of
$\Q_t$ and Lemma \ref{obs:stproduct}, there exists a graph $H$ with
treewidth at most $t$ and at most $n$ vertices such that $G$ is a
subgraph of $H\boxtimes P_n$.
By Lemma \ref{lem:twpw}, $H$ has pathwidth at most $(t+1) \lfloor \log_3 (2n+1)  + 1 \rfloor - 1 = \omega - 1$.
Hence, there exists an interval graph $I$ with clique number
at most $\omega$ containing $H$ as a spanning subgraph. In
particular, $I$ has at most $n$ vertices. By
Lemma \ref{lem:intervaluniversal}, $I$ (and thus $H$) is a subgraph of
$C_{\lceil \log n \rceil} \boxtimes K_\omega$. It follows that $G$ is
a subgraph of $C_{\lceil \log n \rceil} \boxtimes P_n \boxtimes
K_\omega$. This proves that $G_n'$ is indeed universal for
the class of $n$-vertex graphs of $\Q_t$, as desired.

\medskip

The final step consists in reducing the number of vertices in our
universal graph from $n^{1+o(1)}$ to $(1+o(1))n$.
We consider our universal graph $G'_n=G_n \boxtimes K_\omega$ for the family of
$n$-vertex planar graphs, with $N_0=n^{1+o(1)}$ vertices and
$n^{1+o(1)}$ edges. Take $\epsilon=\log^{-1} n$, and let
$k=N_0/(1+\epsilon)n=n^{o(1)}$. By Lemma \ref{lem:alon} there exist
$d=O(\tfrac1{\epsilon}\log n)=O(\log^2 n)$ and a bipartite graph $H$ with partite sets
$V\supseteq V(G_n')$ and $U$, with $|V|=N\le N_0+k$ and $|U|=N/k\le
(1+\epsilon)n+1$,
such that the vertices of $V$ have degree at most $d$ in $H$ and every
$n$-vertex subset of $V$ is saturated by a matching in $H$.

We define a graph $H_n$ from $G_n'$ and $H$ as follows: the
vertex set of $H_n$ is $U\subseteq V(H)$, and two vertices  $u,u'$ are
adjacent in $H_n$ if there are $v,v'\in V=V(G_n')$ such that
$vv'\in E(G_n')$, $vu\in E(H)$, and  $v'u'\in E(H)$. Note that $H_n$ has
at most $d^2|E(G_n')|=O(\log^4 n)\cdot n^{1+o(1)} =n^{1+o(1)}$ edges and $|U|\le
(1+\epsilon)n+1=n+O(n/\log n)=(1+o(1))n$ vertices.

It remains to prove that $H_n$ contains all $n$-vertex graphs of
$\Q_t$ as
subgraphs. Take an $n$-vertex graph $F\in \Q_t$. Then $F$ is a subgraph
of $G_n'$, so there is a set $X$ of $n$ vertices of $G_n'$ such that $F$
is a subgraph of $G_n'[X]$. By Lemma \ref{lem:alon},
there is a matching between $X$ and $N_H(X)$ in $H$ that saturates
$X$. The intersection of this matching with $U$ consists of a set $Y$
of $n$
vertices, and it follows from the definition of $H_n$ that $F$ is a
subgraph of $H_n[Y]$, as desired.
\end{proof}

In the remainder of \cref{sec:universal}, we prove Theorem \ref{thm:main_technical}.

\subsection{Definition of the universal graphs}
\label{sec:universal_graph}

Let $n$ be a positive integer.
We define a graph $G_n$ that will be universal for $n$-vertex subgraphs of $C_{\lceil \log n \rceil}\boxtimes P_n$.
For convenience, let $d:= \lceil \log n \rceil$. Let
$k:=\max(5,\lceil\sqrt{\log n / \log\log n}\rceil)$, as in Lemma \ref{lem:tree_sequence}.
With a slight abuse of notation, let $\lambda:=\lambda(n)$, where
$\lambda(n)$ is the function from Observation \ref{obs:lambda}.

\smallskip

The vertices of the graph $G_{n}$ are all the triples $(x,y,z)$ where
$x,y\in\{0,1\}^*$ are bitstrings such that $|x|+|y|\le d+\lambda+2$
and $z$ is an integer with $z\in\{0,\ldots,d\}$.
When defining the edge set of $G_{n}$, it will be convenient to orient the edges to simplify the discussions later on, the graph $G_n$ itself is of course undirected.
Given two distinct vertices $(x_1,y_1,z_1), (x_2,y_2,z_2)$, we put a directed edge from  $(x_1,y_1,z_1)$ to $(x_2,y_2,z_2)$ if one of the following conditions is satisfied:

\begin{enumerate}[(1)]
\item \label{E1} $y_1=y_2$ and $x_2 \preceq x_1$,
\item \label{E2} $y_1\neq y_2$, and
  \begin{enumerate}[(a)]
    \item \label{y2L} $y_2 \in L(y_1,d+2)$, where $L$ is defined in Observation~\ref{obs:successor-encoding}, and
    \item \label{x2prime} there exists $x_2' \in \{0,1\}^{*}$ with $|x_2'| \leq d + \lambda +2 -|y_1|$ such that $x_1 \comparable  x_2'$, and
    \item \label{nu} there exists $\nu\in \{0,1\}^{*}$ with $|\nu|
      \leq \lambda$ such that $B(x_2',\nu)=x_2$, where $B$ is the
      function from Lemma \ref{lem:tree_sequence}.
    \end{enumerate}
\end{enumerate}

Observe that the third coordinate of the triples is not used when defining adjacencies in $G_n$.
It will be used when proving the universality of $G_n$. Note also that
the definition of our universal graph is explicit, as the functions
$L$ and $B$ are explicit themselves (see the discussion at the end of Section~\ref{sec:bst}).

\smallskip

We start by bounding the number of vertices and edges in $G_n$.

\begin{lemma}
\label{lem:Gn_bound}
The following bounds hold:
\begin{itemize}
    \item $|V(G_{n})| \le 2^{d+\lambda+3}\cdot (d+\lambda+3)^2 \le n\cdot 2^{O(\sqrt{\log n \cdot \log \log n})}$, and
    \item $|E(G_{n})|\le 2^{d+2\lambda +4}\cdot (d+\lambda+3)^6 \le n\cdot 2^{O(\sqrt{\log n \cdot \log \log n})}$.
\end{itemize}
\end{lemma}

\begin{proof}
For each $0\le r \le d+\lambda+2$, there are $(r+1)2^r$ pairs $(x,y)$ with
$x,y\in\{0,1\}^*$ such that $|x|+|y|=r$. It follows that for each
$z\in \{0,\ldots,d\}$, $G_n$ contains at most
\[
\sum_{r=0}^{d+\lambda+2} (r+1)\cdot 2^r \le (d+\lambda+3)2^{d+\lambda+3}
  \]
  vertices of the form $(x,y,z)$.
It follows that  $|V(G_{n})| \le 2^{d+\lambda+3}\cdot (d+\lambda+3)(d+1) \le 2^{d+\lambda+3}\cdot (d+\lambda+3)^2$.

    In order to bound $|E(G_{n})|$, we will bound the number of {\em
      outgoing} edges from a given vertex $(x_1,y_1,z_1)$ of $G_n$.

    The number of choices for $(x_2,z_2)$ that result in an edge of Type \eqref{E1} is at most $(|x_1|+1)\cdot(d+1) \le (d+\lambda+3)(d+1)$. It follows that the number of edges of Type \eqref{E1} is at most
    \[
        |V(G_n)|\cdot (d+\lambda+3)(d+1)\le 2^{d+\lambda+3}\cdot (d+\lambda+3)^4
    \]

  To count outgoing edges of Type \eqref{E2} we again fix $(x_1,y_1,z_1)$ with
  $|x_1|+|y_1|=r$. By Observation~\ref{obs:successor-encoding}, the number of
  choices for  $y_2 \in L(y_1,d+2)$ is at most $d+3$.  The number
  of choices for $x_2'$ is at most
  \[ |x_1|+1+2^{d+\lambda+2-|y_1|-|x_1|}=|x_1|+1+2^{d+\lambda+2-r}\le
    d+\lambda+3+2^{d+\lambda+2-r}\le (d+\lambda+3)2^{d+\lambda+2-r}.
  \]
  The number of choices of $\nu$ is at most $2^{\lambda+1}$.  The choices of $x_2'$ and $\nu$ determine $x_2$.  The number of choices for $z_2$ is $d+1$.  As before, the number of vertices $(x_1,y_1,z_1)$ with $|x_1|+|y_1|=r$ is $(r+1)\cdot 2^r\cdot (d+1)$.  Therefore, the total number of edges of Type \eqref{E2} is at most
    \[
        \sum_{r=0}^{d+\lambda+2} (r+1)\cdot 2^r\cdot (d+1)
            \cdot 2^{\lambda +1}\cdot (d+3)(d+\lambda+3) \cdot 2^{d+\lambda+2-r}\cdot (d+1) \le 2^{d+2\lambda+3}\cdot (d+\lambda+3)^6 \enspace .
    \]
We obtain that the total number of edges in $G_n$ is at most $2^{d+2\lambda +4}\cdot (d+\lambda+3)^6$.
\end{proof}

\subsection{Proof of universality}

\begin{lemma}
\label{lem:Gn_universal}
   The graph $G_{n}$ is universal for the class of $n$-vertex subgraphs of $C_{\lceil \log n \rceil}\boxtimes P_n$.
\end{lemma}

\begin{proof}
    Let $d:= \lceil \log n \rceil$, $k:=\max(5,\lceil\sqrt{\log n / \log\log n}\rceil)$, and $\lambda:=\lambda(n)$.
    Let $G$ be an $n$-vertex subgraph of $C_d\boxtimes P_n$.
    By Lemma \ref{obs:stproduct}, we may assume that $G$ is a subgraph of $C_d\boxtimes P_h$ for some integer $1\le h\le n$ and that, for each $i\in\{1,\ldots,h\}$, there exists at least one vertex $v$ in $C_d$ such that $(v,i)\in V(G)$.  Clearly, it suffices to prove the result when $G$ is an induced subgraph of $C_d\boxtimes P_h$.

    We first define the embedding of $G$ onto $G_n$.  For each $i\in\{1,\ldots,h\}$, let $S_i:=\{v\in V(C_d): (v,i)\in V(G)\}$.
    Recall that $V(C_d) = \{1, \dots, 2^{d+1} -1\}$, thus $S_i \subset \R$.
    Let $T_1,\ldots,T_h$ be the sequence of binary search trees
    obtained by applying Lemma \ref{lem:tree_sequence} to the sequence $S_1, \dots, S_h$.
    Let $T$ be a binary search tree with $V(T)=\{1,\ldots,h\}$
    obtained by applying Lemma \ref{lem:biased-bst} with the weight function $w(i)=|V(T_i)|$,
    for each $i\in\{1,\ldots,h\}$. Let $\varphi:V(C_d)\to\{0,\ldots,d\}$ be a proper colouring of $C_d$.
    (For instance, one could set $\varphi(v) := d_{B_d}(v)$.)
    Each vertex $(v,i)$ of $G$ maps to the vertex
    \[\zeta(v,i):= (\sigma_{T_i}(x_{T_i}(v)), \sigma_T(i),
      \varphi(v)).
      \]

    First we verify that $\zeta$ does indeed take vertices of $G$ onto vertices of $G_n$.
    Let $(v,i)$ be a vertex of $G$ and let $\zeta(v,i)= (x:=\sigma_{T_i}(x_{T_i}(v)), y:=\sigma_T(i), z:=\varphi(v))$.  Clearly, $z\in\{0,\ldots,d\}$.
    Note that $\sum_{j=1}^h w(j) \leq 4n$, by Lemma \ref{lem:tree_sequence}.
    Thus, by Lemma \ref{lem:biased-bst} we have $|y|\le \log (4n) -\log |V(T_i)| \le d+2-\log |V(T_i)|$.
    By Lemma \ref{lem:tree_sequence} (complemented by Observation \ref{obs:lambda}),
    $\h(T_i)\le \log |V(T_i)|+\lambda$  and since
    $|x|=|\sigma_{T_i}(x_{T_i}(v))|\le \h(T_i)$, we have $|x|+|y|\le d+\lambda+2$.
    Thus $(x,y,z)$ is indeed a vertex of $G_n$.

    Next we verify that $\zeta:V(G)\to V(G_n)$ is injective.  Let
    $(v,i)$ and $(w,j)$ be two distinct vertices of $G$.  If $i\neq j$
    then $\sigma_T(i)\neq\sigma_T(j)$.  We may thus assume that $i=j$,
    so $v\ne w$ and both $v$ and $w$ are nodes of $T_i$.  If $x_{T_i}(v)\neq x_{T_i}(w)$ then $\sigma_{T_i}(x_{T_i}(v))\neq\sigma_{T_i}(x_{T_i}(w))$.  We may therefore assume that $x:=x_{T_i}(v)=x_{T_i}(w)$.  This implies that $x\in I_{C_d}(v) \cap I_{C_d}(w)$ so, by Observation~\ref{proper_containment}, $vw\in E(C_d)$.  Since $v\neq w$, this implies that $z_1=\varphi(v)\neq \varphi(w)=z_2$. Thus, $\zeta(v,i)\neq \zeta(w,j)$ for $(v,i)\neq (w,j)$, so $\zeta$ is injective.

    Finally we need to verify that, for each edge $(v,i)(w,j)\in E(G)$, $G_n$ contains the edge $\zeta(v,i)\zeta(w,j)$.  Let $(x_1,y_1,z_1):=\zeta(v,i)$ and let
    $(x_2,y_2,z_2):=\zeta(w,j)$.  There are two cases to consider:

    {\bf Case~1: $j=i$.}
    In this case, $y_1=y_2=\sigma_T(i)$, $v\neq w$ and $vw\in E(C_d)$,
    and $v,w\in V(T_i)$.  By Lemma \ref{consistent_ancestor}, $x_1=\sigma_{T_i}(x_{T_i}(v)) \comparable  \sigma_{T_i}(x_{T_i}(w))=x_2$.  Therefore, $\zeta(v,i)\zeta(w,j)\in E(G_n)$ since it is included in $G_n$ as an edge of Type~\eqref{E1}.

    {\bf Case~2: $j=i+1$.}
    In this case, $y_2 \in L(y_1, \h(T))$ by Observation \ref{obs:successor-encoding}.
    Lemma \ref{lem:biased-bst} ensures that $\h(T) \le \max\{d_T(i):i\in\{1,\ldots,h\}\} = \max\{\log(4n/w(i)):i\in\{1,\ldots,h\}\} \le \log (4n) \le d+2$.
    Thus $y_2 \in L(y_1, d+2)$, and so condition~\eqref{y2L} for edges of Type~\eqref{E2} is satisfied.

    Next, let $x_2':=\sigma_{T_i}(x_{T_i}(w))$. Observe that $w\in
    V(T_i)$ by property~\eqref{pretend_appear} of Lemma
    \ref{lem:tree_sequence}, so $x_2'$ is well defined.  Since
    $(v,i)(w,j)\in E(G)$, either $v=w$ or $vw\in E(C_d)$.  In the
    former case we immediately have $x_2'=x_1$, and thus
    $x_2'\comparable  x_1$ holds.  In the latter case, Lemma \ref{consistent_ancestor} implies that $x_2'=\sigma_{T_i}(x_{T_i}(w)) \comparable  \sigma_{T_i}(x_{T_i}(v))=x_1$.  Therefore $x_2'$ satisfies condition~\eqref{x2prime} for edges of Type~\eqref{E2}.

    Next, by the definition of $\lambda$ and by
    property~\eqref{bound_nu}  of Lemma \ref{lem:tree_sequence}
    (complemented by Observation \ref{obs:lambda}), there exists a bitstring $\nu$ of length at most $\lambda$ such that $B(x_2',\nu)=B(\sigma_{T_i}(x_{T_i}(w)),\nu)=\sigma_{T_{i+1}}(x_{T_{i+1}}(w))=x_2$.  Hence, $x_2'$ and $\nu$ satisfy condition~\eqref{nu} for edges of Type \eqref{E2}.  Therefore, $\zeta(v,i)\zeta(w,j)\in E(G_n)$ since it is included in $G_n$ as an edge of Type~\eqref{E2}.
\end{proof}

\cref{thm:main_technical} follows from Lemmas \ref{lem:Gn_bound} and
\ref{lem:Gn_universal}.

\section{Induced-universal graphs}\label{sec:induced-universal}

In this section we prove \cref{thm:main_induced}.  We describe a graph $U_n$ that is induced-universal for $n$-vertex members of $\mathcal{Q}_t$ and has $n\cdot 2^{O(\sqrt{\log n\log\log n})}\cdot(\log n)^{O(t^2)}$ edges and vertices.  The construction of $U_n$ relies on a relationship between induced-universal graphs and adjacency labelling schemes, which we now describe. Throughout this section, for the sake of brevity, we use $n^{o(1)}$ factors in place of more precise quantities like $2^{O(\sqrt{\log n\log\log n})}$ and (for constant $t$) $(\log n)^{O(t^2)}$.  At the end of this section we give a brief discussion of how the precise result in \cref{thm:main_induced} appears.

\citet{AdjacencyLabellingPlanarJACM} describe a $(1+o(1))\log n$-bit
\emph{adjacency labelling scheme} for graphs in $\mathcal{Q}_t$.  This
means that there is a single function $A:\{0,1\}^*\times\{0,1\}^*
\to\{0,1\}$ such that, for any $n$-vertex graph $G\in\mathcal{Q}_t$
there is an injective labelling $\ell_G:V(G)\to\{0,1\}^{(1+o(1))\log n}$ for which $A(\ell_G(v),\ell_G(w))=1$ if and only if $vw\in E(G)$.  The existence of such a labelling scheme has the following immediate consequence: For every positive integer $n$, there exists a graph $I_n$ having $n^{1+o(1)}$ \emph{vertices} such that, for every $n$-vertex graph $G\in\mathcal{Q}_t$, $I_n$ contains an induced subgraph isomorphic to $G$.  To see this, let $I_n$ be the graph with vertex set $V(I_n):=\{0,1\}^{(1+o(1))\log n}$ and for which $xy\in E(I_n)$ if and only $A(x,y)=1$.  Then, for any $n$-vertex graph $G\in\mathcal{Q}_t$, the induced subgraph $G':=I_n[\{\ell_G(v):v\in V(G)\}]$ is isomorphic to $V(G)$ (and $\ell_G$ gives the isomorphism from $G$ into $G'$).

In \cref{review} we begin by reviewing the adjacency labelling scheme of \citet{AdjacencyLabellingPlanarJACM}. In \cref{density-lower-bound} we show that the induced-universal graph $I_n$ defined in the previous paragraph has $\Omega(n^2)$ edges.  In \cref{modifications} we show how the adjacency labelling scheme can be modified so that the resulting induced-universal graph $U_n$ has $n^{1+o(1)}$ edges.

\subsection{Review of Adjacency Labelling}
\label{review}

In this section we review the adjacency labelling scheme in \cite{AdjacencyLabellingPlanarJACM}.  This review closely follows the presentation of \cite{AdjacencyLabellingPlanarJACM} with a few exceptions that we discuss in footnotes when they occur.  The main purpose of this review is to focus on a list of properties (P1)--(P6) that allow the adjacency labelling scheme to work correctly. Later, we will modify this labelling scheme and show that the modified scheme also has (suitably modified versions of) properties (P1)--(P6).

A \emph{$t$-tree} $H$ is a graph that is either a clique on $t+1$ vertices or contains a vertex $v$ of degree $t$ that is part of a $(t+1)$-clique and such that $H-\{v\}$ is a $t$-tree.  This definition implies that every $t$-tree $H$ has a \emph{construction order} $v_1,\ldots,v_n$ of its vertices such that $v_1,\ldots,v_t$ form a clique and, for each $i\in\{1,\ldots,n\}$, $v_i$ is adjacent to exactly $\min\{i-1,t\}$ vertices among $v_1,\ldots,v_{i-1}$ and these vertices form a clique.

Fix a construction order $v_1,\ldots,v_n$ of $H$ and define
\[C_{v_i}:=\{v_i\}\cup\big\{v_j:v_iv_j\in E(H),\,
  j\in\{1,\ldots,\max(t+1,i)\}\big\}
\]
for each $i\in\{1,\ldots,n\}$.  Then the vertices in $C_{v_i}$ form a clique of order $t+1$ in $H$ that we call the \emph{family clique} of $v_i$. For each $v\in V(H)$, each vertex $w\in C_v$ is called an \emph{$H$-parent} of $v$.  A vertex $a$ of $H$ is an \emph{$H$-ancestor} of $v$ if $a=v$ or $a$ is an $H$-ancestor of some $H$-parent of $v$.  Note that $v$ is an $H$-parent and an $H$-ancestor of itself.

The construction order $v_1,\ldots,v_n$ implies that every $t$-tree $H$ has a proper colouring using $t+1$ colours.  Fix such a colouring $\varphi:V(H)\to\{1,\ldots,t+1\}$. For any vertex $v$ of $H$, the \emph{$i$-parent} of $v$, denoted by $p_i(v)$,  is the unique node $w\in C_v$ with $\varphi(w)=i$.  Note that $v$ is the $\varphi(v)$-parent of itself, i.e, $p_{\varphi(v)}(v)=v$.

It is well known that every graph of treewidth at most $t$ is a subgraph of some $t$-tree.  Thus it is sufficient to describe how the adjacency labelling scheme in \cite{AdjacencyLabellingPlanarJACM} works for any $n$-vertex subgraph $G$ of $H\boxtimes P$ where $H$ is a $t$-tree $H$ and $P$ is a path.  Without loss of generality, we may assume that the vertices of $P$ are the integers $1,\ldots,h$ in the order they occur on the path $P$ and that, for each $y\in\{1,\ldots,h\}$ there exists at least one $v\in V(H)$ such that $(v,y)\in V(G)$, so $h\le n$. Similarly, we may assume that $|V(H)|\le n$.\footnote{This assumption requires that $n\ge t+1$.  We ignore the graphs in $\mathcal{Q}_t$ having fewer than $t+1$ vertices since there are only $O(2^{\binom{t}{2}})$ such graphs.}

The adjacency labelling scheme in \cite{AdjacencyLabellingPlanarJACM}
makes use of an \emph{interval supergraph} of $H$.  Each vertex $v$ of
$H$ is mapped to a real interval $[a_v,b_v]$ in such a way that $vw\in
E(H)$ implies that $[a_v,b_v]\cap [a_w,b_w]\neq\emptyset$.
Lemma \ref{lem:twpw} essentially says that this mapping is \emph{thin}, in the following sense:

\begin{property}
    \item for any $x\in \R$, $|\{v\in V(H): x\in[a_v,b_v]\}|\in O(t\log n)$.\label{thin}
\end{property}

For each $y\in\{0,\ldots,h+1\}$, let $L_y:=\{v\in V(H): (v,y)\in V(G)\}$ and let $S_y:=\bigcup_{v\in L_y}C_v$.  The labelling scheme first finds sets $S^+_1,\ldots,S^+_h$ of total size $O(n)$ such that $S^+_y\supseteq S_{y-1}\cup S_y\cup S_{y+1}$.\footnote{The original labelling scheme only uses $S^+_y\supseteq S_{y-1}\cup S_y$ but it is convenient for us to include $S_{y+1}$ as well and this change does not invalidate anything in the original scheme.}

The adjacency labelling scheme uses a sequence of binary search trees
$T_1,\ldots,T_h$ such that, for each $y\in\{1,\ldots,h\}$ and each
$v\in S^+_y$, $T_y$ contains at least one value $x\in [a_v,b_v]$.
($T_1,\ldots,T_h$ form a \emph{bulk tree sequence} as defined in Lemma
\ref{lem:tree_sequence}, that also plays a central role in the proof of \cref{thm:main}.) This leads to the following very important definition: For each $v\in S^+_y$, $x_{y}(v)$ is the minimum-depth node $x$ of $T_y$ such that $x\in [a_v,b_v]$. Note that $x_y(v)$ is well-defined since $T_y$ contains at least one node $x\in[a_v,b_v]$.   The following property follows from these definitions and Helly's Theorem:\footnote{Helly's Theorem (in 1 dimension): Any finite set of pairwise intersecting intervals has a non-empty common intersection.}

\begin{property}\setcounter{enumi}{1}
    \item For any $v\in L_{y-1}\cup L_y\cup L_{y+1}$, there exists a path $P_y(v)$ that begins at the root of $T_y$ and contains every node in $X_y(v):=\{x_{y}(w): w\in C_v\}$.\label{clique-path}
\end{property}

For each $y\in\{1,\ldots,h\}$ and each $v\in L_y$, we define $P_y(v)$ to be the minimum length path in $T_y$ that satisfies \pref{clique-path}, so that $P_y(v)$ begins at the root of $T$ and ends at the node in $X_y(v)$ of maximum $T_y$-depth. For each $y\in\{1,\ldots,h\}$ and each $x\in V(T_y)$, $d_y(x)$ denotes the depth of $x$ in the tree $T_y$.

It is helpful to think of $x_y$ as a function $x_y:S^+_y\to V(T_y)$.  For each $y\in\{1,\ldots,h\}$ and each node $x$ of $T_y$, let $B_{y,x}:=\{v\in S^+_y: x_y(v)=x\}=x_y^{-1}(x)$.  Since $x\in[a_v,b_v]$ for each $v\in B_{y,x}$, \pref{thin} implies the following property:

\begin{property}\setcounter{enumi}{2}
    \item For each $y\in\{1,\ldots,h\}$ and each $x\in V(T_y)$, $|B_{y,x}|\in O(t\log n)$.\label{small-bags-i}
\end{property}

Recall that, for any node $x$ in a binary search tree $T$, $\sigma_T(x)$ is the binary string $b_1,\ldots,b_k$ obtained from the root-to-$x$ path $x_0,\ldots,x_k$ in $T$ by setting $b_i=0$ or $b_i=1$ depending on whether $x_i$ is the left or right child of $x_{i-1}$, respectively. Note that the function $\sigma_T:V(T)\to\{0,1\}^*$ is injective.  We extend this notation to paths in $T$ so that, if $P$ is a path from the root of $T$ to some node $x$, then $\sigma_T(P):=\sigma_T(x)$.  We will use $\sigma_y$ as a shorthand for $\sigma_{T_y}$.

Let $\psi_y:S^+_y\to\{1,\ldots,O(t\log n)\}$ be a colouring of $S^+_y$ such that, for each $x\in V(T_y)$ and each distinct pair $v,w\in B_{y,x}$, $\psi_y(v)\neq\psi_y(w)$.  Such a colouring exists by \pref{small-bags-i} and because $x_y$ is a function, so each $v\in S^+_y$ appears in $B_{y,x}$ for exactly one $x\in V(T_y)$.  Note that, for any $v\in S^+_y$, the pair $(x_y(v), \psi_y(v))$ uniquely identifies $v$.  Since the signature function $\sigma_y:=\sigma_{T_y}$ is injective, this means that the pair $(\sigma_y(x_y(v)),\psi_y(v))$ also uniquely identifies $v$:

\begin{property}\setcounter{enumi}{3}
    \item For any $y\in\{1,\ldots,h\}$ and any $v,w\in S^+_y$, $v=w$ if and only if $\sigma_y(x_y(v))= \sigma_y(x_y(w)) $ and $\psi_y(v)=\psi_y(w)$.\label{unique-match}
\end{property}

The binary search tree sequence $T_1,\ldots,T_h$  has two additional
properties that are crucial:
\begin{property}\setcounter{enumi}{4}
    \item For each $y\in\{1,\ldots,h\}$, $T_y$ has height $\h(T_y)\le \log|S^+_y|+o(\log n)$. \label{tree-height}
    \item There exists a universal function $J:\{0,1\}^*\times\{0,1\}^*\to\{0,1\}^*$ such that for each $y\in\{1,\ldots,h-1\}$ and each $v\in S^+_y\cap S^+_{y+1}$, there exists a bitstring $\mu_y(v)$ of length $o(\log n)$ such that $J(\sigma_{y}(x_y(v)),\mu_y(v))=\sigma_{y+1}(x_{y+1}(v))$.
    \label{transition-code-v}
\end{property}
Note that \pref{transition-code-v} corresponds precisely to Lemma~\ref{lem:tree_sequence}\eqref{bound_nu}. The bitstring $\mu_y(v)$ is called a \emph{transition
  code}.\footnote{Our presentation here differs slightly from that in
  \cite{AdjacencyLabellingPlanarJACM}.  In
  \cite{AdjacencyLabellingPlanarJACM}, the transition code is used to
  take $\sigma(P_y(v))$ onto $\sigma(P_{y+1}(v))$.  However, the proof
  that this is possible
  \cite[Section~5.3]{AdjacencyLabellingPlanarJACM} uses the existence
  of the transition code described in \pref{transition-code-v} for
  each $w\in C_v$ and the fact that
  $\sigma(P_{y+1}(v))=\sigma(x_{y+1}(w))$ for some $w\in C_v$.\label{foot:P6}}

\subsubsection{The Labels}
\label{labels-i}

For each vertex $(v,y)$ of $G\subseteq H\boxtimes P$, the label $\ell_G(v,y)$ has these parts:

\begin{compactenum}[(L1)]
    \item $\alpha(y)$: a bitstring of length of $\log n-\log |S^+_y|+o(\log n)$.  Given $\alpha(y_1)$ and $\alpha(y_2)$ for any $y_1,y_2\in\{1,\ldots,h\}$, it is possible to distinguish between the following cases:
    \begin{inparaenum}
        \item $y_1=y_2$;
        \item $y_1=y_2+1$;
        \item $y_1=y_2-1$; and
        \item $|y_1-y_2|\ge 2$.
    \end{inparaenum}

    \item $\sigma_y(P_y(v))$: this is a bitstring of length at most $\h(T_y)\le \log|S^+_y| + o(\log n)$

    \item $\eta_y(v)$: a bitstring of length $o(\log n)$.  This
      bitstring is designed so that, for any vertex $v\in S^+_y\cap
      S^+_{y+1}$, it is possible to recover $\sigma_{y+1}(P_{y+1}(v))$
      given only $\sigma_y(P_y(v))$ and $\eta_y(v)$.  The existence of
      $\eta_y(v)$ follows easily from the existence of $\mu_y(v)$ in
      \pref{transition-code-v} and from the knowledge of the content of (L5) below.

    \item $\varphi(v)$: the colour of $v$ in the proper colouring of $H$ (a bitstring of length $\lceil\log(t+1)\rceil$).

    \item $d_y(x_y(p_i(v)))$ for each $i\in\{1,\ldots,t+1\}$ (a bitstring of length $O(t\log\log n)$).\label{dxp}

    \item $\psi_{y+b}(p_i(v))$ for each $i\in\{1,\ldots,t+1\}$ and each $b\in\{-1,0,1\}$ (a bitstring of length $O(t\log\log n + t\log t)$).\label{psi}\footnote{This is another place where our presentation differs slightly from that in \cite{AdjacencyLabellingPlanarJACM}.  In \cite{AdjacencyLabellingPlanarJACM}, the information contained in (L4), (L5), and (L6) is spread across several different parts of the label.}

    \item $a_y(v)$: A bitstring of length $3(t+1)$ that indicates, for each $i\in\{1,\ldots,t+1\}$ and each $b\in\{-1,0,1\}$ whether or not $G$ contains the edge with endpoints $(v,y)$ and $(p_i(v),y+b)$.
\end{compactenum}

The label (L1) comes from Observation \ref{obs:successor-encoding} but requires some further explanation.  First we remark that, like all parts of $\ell_G(v,y)$, the string $\alpha(y):=\alpha_G(y)$ depends on both $G$ and $y$. The string $\alpha(y)$ consists of two parts: $\alpha_1(y)$ is a bitstring of length at most $\log n-\log |S^+_y|$ and $\alpha_2(y)$ is a bitstring of length at most $\log\log n+O(1)$. These strings are designed so that there is a universal function $N$, that does not depend on $G$, such that $N(\alpha(y_1))=\alpha_1(y_2)$ if and only if $y_2=y_1+1$. Clearly this makes it possible to distinguish between cases~(a)--(d). It also has the following implication:  For any fixed binary string $\bar{y}_1$ that we interpret as $\alpha(y_1)$ there are at most $2^{\log \log n+O(1)}=O(\log n)$ binary strings that result in case (b). Indeed, these are strings $\bar{y}_2:=a_1\mathbin{\circ}a_2$ (where $\circ$ denotes concatenation of strings) such that $N(\alpha(y_1))=a_1$ and $|a_2|\le \log\log n + O(1)$. The set of such strings turns out to be useful, so we denote it with $L(\alpha(y_1)):=\{N(\alpha(y_1))\mathbin{\circ} s: s\in\{0,1\}^{\log\log n+O(1)}\}$.


\subsubsection{Adjacency Testing}

Given inputs $\ell_G(v_1,y_1)$ and $\ell_G(v_2,y_2)$, the adjacency testing function $A$ uses $\alpha(y_1)$ and $\alpha(y_2)$ to determine which of the following cases applies:
\begin{enumerate}[(a)]
    \item $y:=y_1=y_2$.  For each $i\in\{1,\ldots,t+1\}$, determine if $v_1=p_i(v_2)$ (or \textit{vice-versa}) and, if so, use $a_y(v_2)$ (or $a_y(v_1)$, respectively) to determine if $(v_1,y)$ and $(v_2,y)$ are adjacent in $G$. Specifically, if $v_1=p_i(v_2)$ then one of the bits in $a_y(v_2)$ indicates whether or not $(v_1,y)$ and $(v_2,y)$ are adjacent in $G$. If $v_1\neq p_i(v_2)$ and $v_2\neq p_i(v_1)$ for every $i\in\{1,\ldots,h\}$, then $v_1v_2\not\in E(H)$ and hence $(v_1,y)$ and $(v_2,y)$ are not adjacent in $G\subseteq H\boxtimes P$.

    By \pref{unique-match}, testing if $v_1=p_i(v_2)$, is equivalent to testing if $\sigma_y(x_y(v_1))=\sigma_y(x_y(p_i(v_2)))$ and $\psi_y(v_1)=\psi_y(p_i(v_2))$. We now show that $\ell_G(v_1,y_1)$ and $\ell_G(v_2,y)$ contain enough information to perform this test.
    \begin{compactitem}
        \item We can recover $d_y(x_y(v_1))=d_y(x_y(p_{\varphi(v_1)}(v_1)))$ and using this, recover $\sigma_y(x_y(v_1))$ from $\sigma_y(P_y(v_1))$ and $d_y(x_y(v_1))$.  Next, we can recover $\sigma_y(x_y(p_i(v_2)))$ from $\sigma_y(P_y(v_2))$ and $d_y(x_y(p_i(v_2)))$. This makes it possible to test if $\sigma_y(x_y(v_1))=\sigma_y(x_y(p_i(v_2)))$.
        \item  The colour $\psi_y(v_1)$ can be recovered from $\ell_G(v_1,y_1)$ since $\psi_y(v_1)=\psi_y(p_{\varphi(v_1)}(v_1))$.  The colour $\psi_y(p_i(v_2))$ is stored explicitly in part (L6) of $\ell_G(v_2,y_2)$.  This makes it possible to test if $\psi_y(v_1)=\psi_y(p_i(v_2))$.
    \end{compactitem}
    \item $y:=y_2=y_1+1$.  In this case, recover $\sigma_y(P_y(v_1))$ from $\sigma_{y_1}(P_{y_1}(v_1))$ and $\eta_{y_1}(v_1)$.  At this point, the algorithm proceeds exactly as in the previous case except for two small changes:
    \begin{inparaenum}[(i)]
        \item the value of $\psi_{y}(v_1)=\psi_{y_1+1}(v_1)$ is obtained from (L6); and
        \item in the final step one bit of $a_{y_2}(v_2)$ (L7) is used to check if $(v_2,y_2)$ is adjacent to $(v_1,y_1)=(p_i(v_2),y_2-1)$ in $G$.
    \end{inparaenum}

    \item $y:=y_1=y_{2}+1$. This case is symmetric to the previous case with the roles of $(v_1,y_1)$ and $(v_2,y_2)$ reversed.

    \item $|y_1-y_2|\ge 2$.  In this case $y_1\neq y_2$ and $y_1y_2\not\in E(P)$ and therefore $(v_1,y_1)$ and $(v_2,y_2)$ are not adjacent in $G\subseteq H\boxtimes P$.
\end{enumerate}

\subsection{Edge density of the induced-universal graph $I_n$}
\label{density-lower-bound}

We now explain why the induced-universal graph $I_n$ defined by the labelling scheme in \cite{AdjacencyLabellingPlanarJACM} is not sparse.  It produces a universal graph $I_n$ having $\Omega(n^2)$ edges.  The main issue is the definition of $P_y(v)$ as a path in $T_y$ that contains every node in $X_y(v):=\{x_y(w):w\in C_v\}$.  The problem comes from the fact that there can be nodes in $X_y(v)$ that have much greater $T_y$-depth than $x_y(v)$.  As we will show below, this ultimately leads to a large complete bipartite graph in $I_n$ with sides $L$ and $R$ in which the elements of $L$ all correspond to a single vertex $(v,y)$ of $H\boxtimes P$.  This problem even occurs when $P$ consists of a single vertex and $H$ is a tree.

Consider the tree $H$ illustrated in \cref{bad-example} that consists of a $5$-vertex path $\beta,u,v,w,\alpha$ and a set of $n-5$ leaves.  Exactly half of these leaves are adjacent to $\beta$ and exactly half are adjacent to $\alpha$.  If we root $H$ at $w$ and perform a preorder traversal, we obtain a construction order $v_1,\ldots,v_n$ of $H$ in which $C_w=\{v,w\}$ and $C_a$ contains $a$ and the parent of $a$ for each $a\in V(H)\setminus\{w\}$.

Observe that $H-\{v\}$ has two components each of size exactly
$(n-1)/2$. Therefore, when the vertices of $H$ are mapped onto intervals it is natural to map $v$ onto the dominating interval $[a_v,b_v]:=[1,n]$.  Since $H-\{v\}$ consists of two stars centered at $\alpha$ and $\beta$, it is then natural to have $[a_\alpha,b_\alpha]:=[1, (n-1)/2]$ and $[a_\beta,b_\beta]:=[n/2+1,n-1]$. Now, $H-\{v,\alpha,\beta\}$ has no edges, so the remaining vertices can be mapped to appropriate zero-length intervals. All nodes adjacent to $\alpha$ (including $w$) are mapped to $[i,i]$ for distinct $i\in\{1,\ldots, (n-5)/2\}$. All nodes adjacent to $\beta$ (including $u$) are mapped to $[n/2+j, n/2+j]$ for distinct $j\in\{1,\ldots, (n-5)/2\}$.

\begin{figure}
    \begin{center}
        \includegraphics{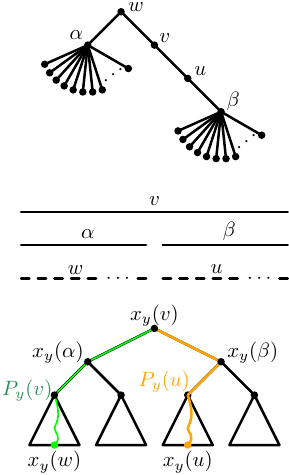}
    \end{center}
    \caption{A tree $H$ that leads to $\Omega(n^2)$ edges in $I_n$.}
    \label{bad-example}
\end{figure}

Let $\alpha_i$ (respectively $\beta_j$) denote the node adjacent to $\alpha$ (respectively, $\beta$) that maps to the interval $[i,i]$ (respectively $[n/2+j,n/2+j]$).  It is entirely possible that $w=\alpha_p$ and $u=\beta_q$ for some $n/12< p,q\le 2n/12$.  Suppose this is the case.  For each $i,j\in\{1,\ldots,n/12\}$, consider the induced subgraph $H_{i,j}$ of $H$ having vertex set $V(H_{i,j})$ that contains
\begin{compactenum}
    \item $\beta,u,v,w,\alpha$;
    \item $\alpha_1,\ldots,\alpha_i$ and $\alpha_{2n/12+1},\ldots,\alpha_{2n/12+n/12-i}$;
    \item $\alpha_{n/4+1},\ldots,\alpha_{n/4+n/12}$;
    \item $\beta_1,\ldots,\beta_j$ and $\beta_{2n/12+1},\ldots,\beta_{2n/12+n/12-j}$;
    \item $\beta_{n/4+1},\ldots,\beta_{n/4+n/12}$;
\end{compactenum}

Let $P_1$ be a path consisting of a single vertex.  If we apply the labelling scheme of \citet{AdjacencyLabellingPlanarJACM} to $H_{i,j}\boxtimes P_1$, to obtain a labelling $\ell_{i,j}:V(H_{i,j})\to\{0,1\}^*$ then the binary search tree $T_{1}$ used in defining $\ell_{i,j}$ could be any balanced binary search tree containing
\begin{compactenum}
    \item a root $r:=n/2$ so that $x_y(v)=r$.
    \item depth-$1$ nodes $a=n/4$ and $b=3n/4$ so that $x_y(\alpha)=a$ and $x_y(\beta)=b$.
    \item $\{k:\alpha_k \in V(H_{i,j})\}$;
    \item $\{n/2+k:\beta_k\in V(H_{i,j})\}$.
\end{compactenum}

The first two levels of $T_1$ are fixed, independent of $i,j$ and each of the four depth-$2$ nodes is the root of a subtree of size exactly $n/12$. In particular, the ``shape'' of $T_{1}$ can be the same for any $i,j\in\{1,\ldots,n/12\}$.  For example, if $n/12=2^k-1$ for some integer $k$, then $T_1$ could be a complete binary tree of height $k+2$.  Suppose that this is the case.  Then $\sigma_1(P_1(u))=\sigma_1(x_1(u))$ depends only on the choice of $j$.  Similarly, $\sigma_1(P_1(v))=\sigma_1(x_1(w))$ depends only on the choice of $i$.

This means that the label $\ell_{i}(v):=\ell_{i,j}(v,1)$ depends only on $i$. Furthermore, for any $i_1\neq i_2$, $\ell_{i_1}(v)\neq\ell_{i_2}(v)$.
Similarly, the label $\ell_j(u):=\ell_{i,j}(u,1)$ depends only on $j$ and is distinct for each $j\in\{1,\ldots,n/12\}$.  Furthermore $uv$ is an edge of $T_{i,j}$ for each $i,j\in\{1,\ldots,n/12\}$, so $A(\ell_{i}(v),\ell_{j}(u))=1$ for each $i,j\in\{1,\ldots,n/12\}$. Therefore, the universal graph $I_n$ contains a complete bipartite subgraph with parts $L:=\{\ell_{i}(v):i\in\{1,\ldots,n/12\}\}$ and $R:=\{\ell_{j}(u):j\in\{1,\ldots,n/12\}\}$.  Therefore $|E(I_n)|\ge n^2/144$.

\subsection{A sparse induced-universal graph}
\label{modifications}

We now describe how to modify the adjacency labelling scheme of
\citet{AdjacencyLabellingPlanarJACM} so that the resulting
induced-universal graph is sparse.  As discussed above, the main
difficulty comes from the fact that, for some vertex $(v,y)\in V(G)$,
$v$ can have an $H$-parent $w$ such that $x_y(w)$ has $T_y$-depth much
greater than $x_y(v)$.  In order to avoid this, we modify the function
$x_y:S^+_y\to V(T_y)$ to create a new function $x'_y$ such that, if $w$ is an $H$-parent of $v$ then $d_y(x'_y(w)) \le d_y(x'_y(v)) + 1$.  This has to be done carefully in order to preserve \pref{clique-path} and \pref{small-bags-i}. Initially $x'_y(v)=x_y(v)$ for each $v\in S^+_y$, but then modifications are performed by calling the following recursive procedure with the root of $T_y$ as its argument:

\medskip
\noindent
\begin{minipage}{\textwidth}
    $\textsc{Fixup}(x)$:
    \begin{algorithmic}[1]
        \FOR{each $v\in S^+_y$ such that $x'_y(v)=x$}
            \FOR{each $w\in C_v \cap S^+_y$}
                \IF{$d_y(x'_y(w)) > d_y(x)+1$}
                    \STATE{\COMMENT{\emph{this implies that $x'_y(w)=x_y(w)$}}}
                    \STATE{$x'_y(w)\gets\mbox{the depth-$(d_y(x)+1)$ $T_y$-ancestor of $x'_y(w)$}$\label{changes}}
                    \STATE{\COMMENT{\emph{so $x'_y(w)$ becomes a child of $x=x'_y(v)$}} }
                \ENDIF
            \ENDFOR
        \ENDFOR
        \STATE{\textsc{Fixup}(left child of $x$) (if any)}
        \STATE{\textsc{Fixup}(right child of $x$) (if any)}
    \end{algorithmic}
\end{minipage}
\medskip


Observe that the only modifications to $x'_y$ occur in Line~\ref{changes} and they involve setting $x'_y(w)$ to a $T_y$-ancestor of $x'_y(w)$.  For each $v\in S^+_y$, $x'_y(v)=x_y(v)$ before the algorithm runs.  Therefore, after the algorithm runs to completion, $x'_y(v)$ is a $T_y$-ancestor of $x_y(v)$.  This ensures that \pref{clique-path} holds for $x'_y$.  Furthermore, Lines~3--6 of the algorithm ensure that, for any $H$-parent $w$ of $v$, $d_y(x'_y(w))\le d_y(x'_y(v))+1$.  Therefore, after running $\textsc{Fixup}(r)$, the following strengthening of \pref{clique-path} holds:

\begin{pproperty}\setcounter{enumi}{1}
    \item For any $v\in L_{y-1}\cup L_y\cup L_{y+1}$, there exists a path $P'_y(v)$ of length at most $d_y(x_y(v))+1$ that begins at the root of $T_y$ and contains every node in $X'_y(v):=\{x'_{y}(w): w\in C_v\}$.\label{clique-path-ii}
\end{pproperty}

Property~\pref{clique-path} is one of two critical properties needed
by the function $x_y$.  The other, \pref{small-bags-i}, bounds the
size of $B_{y,x}:=\{v\in S^+_y: x_y(v)=x\}$ by $O(t\log n)$.  However,
it is not the case that $x'_y$ satisfies \pref{small-bags-i}.  Indeed,
$B'_{y,x}:=\{v\in S^+_y: x'_y(v)=x\}$ can be much larger than
$B_{y,x}$, and even larger than $O(t\log n)$.  Nevertheless, the next lemma shows that, for fixed $t$, the size of $B'_{y,x}$ remains polylogarithmic in $n$.

\begin{lemma}\label{small-bags-ii-lem}
    For each $y\in\{1,\ldots,h\}$ and each node $x$ of $T_y$, $|B'_{y,x}|\in O(t(\log n)^{t+2})$.
\end{lemma}

\begin{proof}
    Let $x$ be some node of $T_y$ and suppose that $x'_y(w)=x$ for
    some $w\in S^+_y$. We now define a path $w_0,w_1,w_2,\ldots,w_d$ in $H$ by the following procedure. We start with
    $w_0=w$. At each step    $i\ge 0$, we first check whether
    $x_y(w_i)=x'_y(w_i)$, and, if so, we set
    $d:=i$ and stop the process. Otherwise, it means that the
    definition of $x'_y(w_i)$ was modified at some point by $\textsc{Fixup}$, and thus $w_i$ has a neighbor $w_{i+1}$ in $H$ with $w_i\in C_{w_{i+1}}$, such that
    $x'_y(w_{i})$ was set to be a child of
    $x'_y(w_{i+1})$ in $T_y$ by $\textsc{Fixup}$. In this way, we obtain a path $w_0,w_1,w_2,\ldots,w_d$, $d\ge 0$, in $H$ such that
    \begin{compactenum}[(a)]
        \item $w_0=w$;
        \item $w_{i-1}$ is an $H$-parent of $w_i$ for each $i\in\{1,...,d\}$;
        \item $x'_y(w_{i})$ is the $T_y$-parent of $x'_y(w_{i-1})$ for each $i\in\{1,...,d\}$; and
        \item $x_y(w_d)=x'_y(w_d)$.
    \end{compactenum}
    In particular, $w$ is an $H$-ancestor of $w_d$ and there is a path $w_0,\ldots,w_d$ in $H$ of length at most $d$ with endpoints $w$ and $w_d$.  In the language of \citet{pilipczuk.siebertz:polynomial-journal} $w_0$ is \emph{$d$-reachable} from $w_d$.  \citet[Lemma~13]{pilipczuk.siebertz:polynomial-journal} show that the number of $d$-reachable $H$-ancestors of any node $v$ in a $t$-tree $H$ is at most $\binom{d+t}{t}$.

    Now, let $x=x_0,\ldots,x_k$ be the path from $x=x_0$ to the root
    $x_k$ of $T_y$. By the preceding argument, for each $w\in
    B'_{y,x}$ there exists some $d\in\{0,\ldots,k\}$ such that $w$ is
    a $d$-reachable $H$-ancestor of some node $v\in B_{y,x_d}$.
    Recall that by \pref{tree-height}, $k\le \h(T_y)=(1+o(1))\log n$, so it follows that
    \[
        |B'_{y,x}|
            \le \sum_{d=0}^k |B_{y,x_d}|\binom{d+t}{t}
            \in O(t\log n\cdot k^{t+1})
            \subseteq O(t(\log n)^{t+2}) \enspace . \qedhere
    \]
\end{proof}

Therefore, by Lemma \ref{small-bags-ii-lem}, $x'_y$ satisfies the following weakening of \pref{small-bags-i}:

\begin{pproperty}\setcounter{enumi}{2}
    \item For each $y\in\{1,\ldots,h\}$ and each $x\in V(T_y)$, $|B'_{y,x}|\in O(t(\log n)^{t+2})$. \label{small-bags-ii}
\end{pproperty}



Let $\psi'_y:S^+_y\to\{1,\ldots,O(t(\log n)^{t+2})\}$ be a colouring of $S^+_y$ such that, for each $x\in V(T)$ and each distinct pair $v,w\in B'_{y,x}$ $\psi'_y(v)\neq\psi'_y(w)$. Such a colouring exists because, by \psref{small-bags-ii}, $|B'_{y,x}|\in O(t(\log n)^{t+2})$ and $x'_y$ is a function, so each $v\in S^+_y$ appears in $B'_{y,x}$ for exactly one $x\in V(T_y)$.  Since $x'_y:S^+_y\to V(T_y)$ is a function and $\sigma_y$ is injective we have the following variant of \pref{unique-match}:

\begin{pproperty}\setcounter{enumi}{3}
    \item For any $y\in\{1,\ldots,h\}$ and any $v,w\in S^+_y$,  $v=w$ if and only if  $\sigma_y(x'_y(v))= \sigma_y(x'_y(w)) $ and $\psi'_y(v)=\psi'_y(w)$.\label{unique-match-ii}
\end{pproperty}

\subsection{The New Labels}

For each vertex $(v,y)$ of $G$, the label $\ell'_G(v,y)$ has these parts:

\begin{enumerate}[(NL1)]
    \item $\alpha(y)$: this is unmodified from the original scheme.

    \item $\sigma_y(x_y(v))$: note that this is not $\sigma_y(P'_y(v))$, but
    $\sigma_y(x'_y(v))$ can be recovered from $\sigma_y(x_y(v))$ and
    $d_y(x'_y(p_{\varphi(v)}(v)))$. This makes it possible to recover
    $\sigma_y(P'_y(v))=\sigma_y(x'_y(v))\mathbin{\circ} r_y(v)$ where
    $r_y(v)$ is defined in \nlref{new-l8}, below (recall that $\circ$
    is used to denote string concatenation).

    \item $\mu_y(v)$: a bitstring of length $o(\log n)$.  This bitstring, defined in \pref{transition-code-v}, is designed so that for any vertex $v\in S^+_y\cap S^+_{y+1}$, it is possible to recover $\sigma_{y+1}(x_{y+1}(v))$ given only $\sigma_y(x_y(v))$ and $\mu_y(v)$.

    \item $\varphi(v)$: the colour of $v$ in the proper colouring of $H$ (a bitstring of length $O(\log t)$).

    \item $d_y(x'_y(p_i(v)))$ for each $i\in\{1,\ldots,t+1\}$ (a bitstring of length $O(t\log\log n)$).

    \item $\psi'_{y+b}(p_i(v))$ for each $i\in\{1,\ldots,t+1\}$ and each $b\in\{-1,0,1\}$ (by \psref{small-bags-ii}, this is a bitstring of length $O(t^2\log\log n)$).\label{psi-prime}

    \item $a_y(v)$: this is unmodified from the original scheme. \label{new-ay}

    \item $r_{y+b}(v)$ for each $b\in\{-1,0,1\}$: Three binary strings, each of length at most 1 such that $\sigma_{y+b}(P'_{y+b}(v))=\sigma_{y+b}(x'_{y+b}(v))\mathbin{\circ}r_{y+b}(v)$ for each $b\in\{-1,0,1\}$. \label{new-l8}
\end{enumerate}

\subsection{Adjacency Testing}

Given inputs $\ell'_G(v_1,y_1)$ and $\ell'_G(v_2,y_2)$, the adjacency testing function $A$ for the new labelling scheme uses $\alpha(y_1)$ and $\alpha(y_2)$ to determine which of the following cases applies:
\begin{enumerate}[(a)]
    \item $y:=y_1=y_2$.  For each $i\in\{1,\ldots,t+1\}$, determine if $v_1=p_i(v_2)$ (or \textit{vice-versa}) and, if so, use $a_y(v_2)$ (or $a_y(v_1)$, respectively) to determine if $(v_1,y)$ and $(v_2,y)$ are adjacent in $G$. Specifically, if $v_1=p_i(v_2)$ then one of the bits in $a_y(v_2)$ indicates whether or $(v_1,y_1)$ and $(v_2,y_1)$ are adjacent in $G$. If $v_1\neq p_i(v_2)$ and $v_2\neq p_i(v_1)$ for every $i\in\{1,\ldots,h\}$, then $v_1v_2\not\in E(H)$ and hence $(v_1,y)$ and $(v_2,y)$ are not adjacent in $G\subseteq H\boxtimes P$.

    By \psref{unique-match-ii}, testing if $v_1=p_i(v_2)$, is equivalent to testing if $\sigma_y(x'_y(v_1))=\sigma_y(x'_y(p_i(v_2)))$ and $\psi'_y(v_1)=\psi'_y(p_i(v_2))$. We now show that $\ell'_G(v_1,y_1)$ and $\ell'_G(v_2,y_2)$ contain enough information to perform this test.
    \begin{compactitem}
        \item We can recover $d_y(x'_y(v_1))=d_y(x'_y(p_{\varphi(v_1)}(v_1)))$ and using this, recover $\sigma_y(x'_y(v_1))$ from $\sigma_y(x_y(v_1))$ and $d_y(x'_y(v_1))$.  Next, we can recover $\sigma_y(x'_y(p_i(v_2)))$ from $\sigma_y(P'_y(v_2))$ and $d_y(x'_y(p_i(v)))$. This makes it possible to test if $\sigma_y(x'_y(v_1))=\sigma_y(x'_y(p_i(v_2)))$.
        \item  The colour $\psi'_y(v_1)$ can be recovered from $\ell'_G(v_1,y_1)$ since $\psi'_y(v_1)=\psi'_y(p_{\varphi(v_1)}(v_1))$.  The colour $\psi'_y(p_i(v_2))$ is stored explicitly in $\ell'_G(v_2,y_2)$.  This makes it possible to test if $\psi'_y(v_1)=\psi'_y(p_i(v_2))$.
    \end{compactitem}

    \item $y:=y_2=y_1+1$.  In this case, recover $\sigma_y(x_y(v_1))$ from $\sigma_{y_1}(x_{y_1}(v_1))$ and $\mu_{y_1}(v_1)$.  Next, recover $\sigma_y(P'_y(v))=\sigma_y(x_y(v_1))\mathbin{\circ}r_{y_1+1}(v)$. At this point, the algorithm proceeds exactly as in the previous case except for two small changes:
    \begin{inparaenum}[(i)]
        \item the value of $\psi_{y}'(v_1)=\psi_{y_1+1}'(v_1)$ is obtained from \nlref{psi-prime}; and
        \item in the final step one bit of $a_{y_2}(v_2)$
          \nlref{new-ay} is used to check whether $(v_2,y_2)$ is adjacent to $(v_1,y_1)=(p_i(v_2),y_2-1)$ in $G$.
    \end{inparaenum}
     %

    \item $y:=y_1=y_2+1$. This case is symmetric to the previous case with the roles of $(v_1,y_1)$ and $(v_2,y_2)$ reversed.

    \item $|y_1-y_2|\ge 2$.  In this case $y_1\neq y_2$ and $y_1y_2\not\in E(P)$ and therefore $(v_1,y_1)$ and $(v_2,y_2)$ are not adjacent in $G\subseteq H\boxtimes P$.
\end{enumerate}

\subsection{Bounding the number of edges}

In the preceding sections we have described an adjacency testing
function $A$ such that, for any $n$-vertex graph $G\in \mathcal{Q}_t$, there
exists an injective labelling $\ell'_G:V(G)\to\{0,1\}^{(1+o(1))\log n}$ such
that, for any $v,w\in V(G)$, $A(\ell'_G(v),\ell'_G(w))=1$ if and only
if $vw\in E(G)$.  We define the induced-universal graph $U_n$ as
follows: $V(U_n)$ contains $\ell'_G(v,y)$ for each $n$-vertex graph
$G\in\mathcal{Q}_t$ and each $(v,y)\in V(G)$. Similarly, an edge
$\ell_1\ell_2$ is in $U_n$ if and only if there exists an $n$-vertex
graph $G\in\mathcal{Q}_t$ that contains an edge $vw$ such that
$\ell'_G(v)=\ell_1$ and $\ell'_G(w)=\ell_2$.  As already discussed, it
follows from the correctness of the labelling scheme that $U_n$ is induced-universal for $n$-vertex graphs in $\mathcal{Q}_t$.

We will now show that $U_n$ has $n^{1+o(1)}$ vertices and edges.  This
analysis mostly follows along the same lines as the analysis of
Section~\ref{sec:universal} but is, by necessity, a little less
modular.\footnote{The modular approach used in
  Section~\ref{sec:universal} to describe a universal graph can be
  ruled out by a simple counting argument.  \Cref{sec:universal}
  describes a universal graph for the class $\mathcal{C}$ of
  $n$-vertex subgraphs of $C_d\boxtimes K_\omega\boxtimes P_n$ for
  $d,\omega\in\Theta(\log n)$.  However, the graph $G:=C_{\log
    n-\log\log n}\boxtimes K_{\log n}$ has $n$ vertices and
  $\Theta(n\log^2 n)$ edges, and lies in $\mathcal{F}$.
  The graph $G$ has at least $2^{\Omega(n\log^2 n)}$
  non-isomorphic $n$-vertex subgraphs, and thus $\mathcal{C}$ contains
  at least $2^{\Omega(n\log^2 n)}$
  non-isomorphic  graphs.  On the other hand, any graph with
  $n^{1+o(1)}$ vertices has at most $\binom{n^{1+o(1)}}{n}$ $n$-vertex
  induced subgraphs and since
$\binom{n^{1+o(1)}}{n} < n^{(1+o(1))n} = 2^{(1+o(1))n\log n}\ll
2^{\Omega(n\log^2 n)}$, it follows that a graph on at most
$n^{1+o(1)}$ vertices cannot be induced-universal for $\mathcal{C}$.}
%
%
%
In this analysis, it will be helpful to think of each label $\ell'_G(v,y)$ in the labelling of a graph $G$ as a triple $(x,\bar{y},z)$ where $x=\sigma_y(x_y(v))$, $\bar{y}=\alpha(y)$, and $z$ is the concatenation of the bitstrings (NL3)--(NL8). Of course, since each vertex of $U_n$ is $\ell'_G(v,y)$ for some $n$-vertex $G\in\mathcal{Q}_t$ and some $(v,y)\in V(G)$, we can also treat the vertices of $U_n$ as triples.  Thus, each vertex of $U_n$ is a triple $(x,\bar{y},z)$ where $x$, $\bar{y}$, and $z$ are bitstrings with $|x|+|\bar{y}|\le \log n + \lambda$, $|z|\le \lambda$, and $\lambda\in o(\log n)$.

In the proofs below, whenever we use
Property~\psref{clique-path-ii} explicitly, what we really use is only the weaker
Property~\pref{clique-path}. So let us first explain where
Property~\psref{clique-path-ii}  is really being used and makes a
crucial difference with the previous labelling scheme with parts (L1)--(L7).  Part~(NL8) of $\ell'_G(v,y)$, which is part of $z$, has constant length and makes it possible to recover $\sigma_y(P'_y(v))$ from (NL2), which has length $d_y(x_y(v))$.  With the original Property~\pref{clique-path}, this would not be possible: recovering $\sigma_y(P_y(v))$ from $\sigma_y(x_y(v))$ requires a string of length $|\sigma_y(P_y(v))|-d_y(x_y(v))$.  In this case, the length of (NL8), and hence the length of $z$ could only be bounded by $h(T_y)-d_y(x_y(v))$ which, as shown in \cref{density-lower-bound}, may be $\Omega(\log n)$.

\begin{lemma}\label{vertex-count}
    The graph $U_n$ has $n^{1+o(1)}$ vertices.
\end{lemma}

\begin{proof}
    Consider a vertex $(x,\bar{y},z)$ of $U_n$. The pair $(x,\bar{y})$ consists of two bitstrings of total length $r := |x|+|\bar{y}|\le\log n + \lambda$.  For a fixed $r$, the number of such $(x,\bar{y})$ is $(r+1)2^{r}$. Therefore, the number of such $(x,\bar{y})$ over all choices of $r$ is
    \[
        \sum_{r=0}^{\log n + \lambda} (r+1)2^r \leq 2^{\log n + \lambda+1}(\log n+\lambda + 1) = n^{1+o(1)} \enspace .
    \]
    The third coordinate, $z$ is a bitstring of length at most $\lambda$. The number of such bitstrings is $2^{\lambda+1}-1=n^{o(1)}$.  Therefore, the number of choices for $(x,\bar{y},z)$ is $n^{1+o(1)}\cdot n^{o(1)}=n^{1+o(1)}$.
\end{proof}

As in Section~\ref{sec:universal}, we distinguish between two kinds of edges in $U_n$.  An edge with endpoints $(x_1,\bar{y}_1,z_1)$ and $(x_2,\bar{y}_2,z_2)$ is a \emph{Type~1} edge if $\bar{y}_1=\bar{y}_2$ and is a \emph{Type~2} edge otherwise.  We count Type~1 and Type~2 edges separately.

\begin{lemma}\label{flat-edges}
    The graph $U_n$ contains $n^{1+o(1)}$ Type~1 edges.
\end{lemma}

\begin{proof}
    Let $(x_1,\bar{y},z_1)(x_2,\bar{y},z_2)$ be a Type~1 edge of $U_n$
    and, for each $i\in\{1,2\}$, let $\ell_i:=(x_i,\bar{y},z_i)$.  Since $\ell_1\ell_2$ lies in $E(U_n)$, there  exists some $t$-tree
    $H$, some path $P$, some $n$-vertex subgraph $G$ of $H\boxtimes
    P$, and some edge $(v_1,y_1)(v_2,y_2)$ of $G$ such that
    $\ell_1=\ell'_G(v_1,y_1)$ and $\ell_2=\ell'_G(v_2,y_2)$. For this graph $G$, $\alpha(y_1)=\bar{y}=\alpha(y_2)$ which implies that $y:=y_1=y_2$ for some integer $y$.

    The existence of the edge $(v_1,y)(v_2,y)$ in $G$ implies the existence of the edge $v_1v_2$ in $H$.  Therefore, $v_1$ is an $H$-parent of $v_2$, or vice-versa. Property~\psref{clique-path-ii} implies that one of $x_1=\sigma_y(x_y(v_1))$ or $x_2=\sigma_y(x_y(v_2))$ is a prefix of the other.  Assume, without loss of generality, that $x_2$ is a prefix of $x_1$ and direct the edge $\ell_1\ell_2$ away from $\ell_1$.  For a fixed $(x_1,\bar{y},z_1)$, the number of $x_2$ that are a prefix of $x_1$ is most $|x_1|+1\le\log n+\lambda+1=n^{o(1)}$. For a fixed $(x_1,\bar{y},z_1)$, the number of $(x_2,\bar{y},z_2)$ in which $x_2$ is a prefix of $x_1$ is at most $n^{o(1)}\cdot 2^{\lambda+1}=n^{o(1)}$.

    Therefore, each vertex $(x_1,\bar{y},z_1)$ of $U_n$ has at most
    $n^{o(1)}$ Type~1 edges directed away from it.  Therefore the
    number of Type~1 edges in $U_n$ is at most $|V(U_n)|\cdot
    n^{o(1)}=n^{1+o(1)}$, where the upper bound on $|V(U_n)|$ comes
    from Lemma \ref{vertex-count}.
\end{proof}

\begin{lemma}\label{vertical-edges}
    The graph $U_n$ contains at most $n^{1+o(1)}$ Type~2 edges.
\end{lemma}

\begin{proof}
    Let $(x_1,\bar{y}_1,z_1)(x_2,\bar{y}_2,z_2)$ be a Type~2 edge of
    $U_n$ and, for each $i\in\{1,2\}$, let
    $\ell_i:=(x_i,\bar{y}_i,z_i)$.  Since $\ell_1\ell_2$ lies in $E(U_n)$, there exists some $t$-tree $H$, some path $P$, some $n$-vertex subgraph $G$ of $H\boxtimes P$, and some edge $(v_1,y_1)(v_2,y_2)$ of $G$ such that $\ell_1=\ell'_G(v_1,y_1)$ and $\ell_2=\ell'_G(v_2,y_2)$.

    Since $\alpha(y_1)=\bar{y}_1\neq \bar{y}_2=\alpha(y_2)$, we have $y_1\neq y_2$. The existence of the edge $(v_1,y_1)(v_2,y_2)$ in $G$ therefore implies that $y_1y_2$ is an edge of $P$ so that (without loss of generality) $y_1=y$ and $y_2=y+1$ for some $y\in\{1,\ldots,h-1\}$.  Now, $\bar{y}_1=\alpha(y)$ and $\bar{y}_2=\alpha(y+1)$.  Specifically $\bar{y}_2\in L(\bar{y}_1)$ (see \cref{labels-i}) and $|L(\bar{y}_1)|\in O(\log n)$.  Therefore, for a fixed $\bar{y}_1$, the number of possible choices for $\bar{y}_2$ is $O(\log n)$.

    The existence of the edge $(v_1,y_1)(v_2,y_2)$ in $G$ implies that $v_1=v_2$ or that $v_1v_2\in E(H)$.
    \begin{enumerate}
        \item If $v_1=v_2$, then $x_2=J(x_1,\mu_y(v_1))$.  Since
          $\mu_y(v_1)$ is included as part of $z_1$ the condition
          $v_1=v_2$ implies that fixing
          $(x_1,\bar{y}_1,z_1)=\ell'_G(v_1,y_1)$ fixes the value of
          $x_2$.  We have already established that, for a fixed
          $\bar{y}_1$, the number of options for $\bar{y}_2$ is
          $O(\log n)$.  Finally, $z_2$ is a bitstring of length at
          most $\lambda$, so the number of options for $z_2$ is at
          most $2^{\lambda+1}-1=n^{o(1)}$.  Therefore, for a fixed
          $(x_1,\bar{y}_1,z_1)$ the number of options for
          $(x_2,\bar{y}_2,z_2)$ in this case is at most
        \[
            1\cdot O(\log n) \cdot n^{o(1)} = n^{o(1)} \enspace .
        \]
        By Lemma \ref{vertex-count}, the number of choices for $(x_1,\bar{y}_1,z_1)$ is at most $n^{1+o(1)}$.  Therefore, the number of Type~2 edges in $U_n$ contributed by edges $(v_1,y_1)(v_2,y_2)$ in $n$-vertex graphs $G\in\mathcal{Q}_t$ where $v_1=v_2$ is at most $n^{1+o(1)}\cdot n^{o(1)} = n^{1+o(1)}$.

        \item If $v_1v_2\in E(H)$ then recall the definition of
          $S^+_y$,  which implies that $v_1,v_2\in S^+_y\cap
          S^+_{y+1}$.  Since $v_1v_2\in E(H)$, at least one of $v_1$
          or $v_2$ is an $H$-parent of the other. Since $(v_2,y+1)\in
          V(G)$, $v_2\in S^+_y$ so $x_y(v_2)$ is defined. By
          \psref{clique-path-ii}, one of $x_2':=\sigma_y(x_y(v_2))$ or
          $x_1=\sigma_y(x_y(v_1))$ is a prefix of the other.  By
          \pref{tree-height}, $|x_2'|\le \h(T_y)\le \log|S^+_y| +
          \lambda \le\log n+\lambda -|\bar{y}_1|$, where the final
          inequality comes from the property of $\alpha$ in (L1) and
          (NL1). Therefore, for a fixed $(x_1,\bar{y}_1,z_1)$, the
          number of choices for $x_2'$ is at most $|x_1|+1+2^{\log n +
            \lambda - |\bar{y}_1|- |x_1|}=n^{1+o(1)}\cdot 2^{-|x_1|-|\bar{y}_1|}$.

        Since $x_2'=x_y(v_2)$, by \pref{transition-code-v}, there exists a bitstring $\mu_y(v_2)$ of length $o(\log n)$ such that $J(x_2',\mu_y(v_2))=\sigma_{y+1}(x_{y+1}(v_2))=x_2$.  Therefore, for a fixed $x_2'$, the number of choices for $x_2$ is at most $2^{o(\log n)}=n^{o(1)}$.  Thus, for a fixed $(x_1,\bar{y}_1,z_1)$, the number of choices for $(x_2,\bar{y}_2,z_2)$ is at most
        \[
            n^{1+o(1)}2^{-|x_1|-|\bar{y}_1|} \cdot O(\log n) \cdot
            n^{o(1)} = n^{1+o(1)}\cdot 2^{-|x_1|-|\bar{y}_1|} \enspace .
        \]
        where the first  factor counts the number of options for
        $x_2$, the second the number of options for $\bar{y}_2\in L(\bar{y}_1)$, and the third the number of options for $z_2$.  For fixed $r:=|x_1|+|\bar{y}_1|$, the number of choices for $(x_1,\bar{y}_1)$ is $(r+1)\cdot 2^{r}$.  Therefore, for a fixed $r$, the number of choices for $(x_1,\bar{y}_1,z_1)$ is $(r+1)\cdot 2^r \cdot(2^{\lambda+1}-1)=2^r\cdot n^{o(1)}$.  We can now sum over $r$ to determine that the total number of Type~2 edges contributed by some edge $(v_1,y_1)(v_2,y_2)$ in some $n$-vertex graph $G\in \mathcal{Q}_t$ with $v_1\neq v_2$ is at most
        \[
            \sum_{r=0}^{\log n+\lambda} 2^r\cdot n^{1+o(1)}\cdot 2^{-r} = n^{1+o(1)}(\log n+\lambda+1) = n^{1+o(1)}.
        \]
    \end{enumerate}
    Each Type~2 edge $(x_1,\bar{y}_1,z_1)(x_2,\bar{y}_2,z_2)$ of $U_n$
    is contributed by some edge $(v_1,y_1)(v_2,y_2)$ in some $n$-vertex graph
    $G\in\mathcal{Q}_t$ such that either $v_1=v_2$ or $v_1\neq v_2$.
    Therefore, the two cases analyzed above establish that $U_n$ has
    $n^{1+o(1)}$ Type~2 edges.
\end{proof}

A more careful handling of $n^{o(1)}$ factors in the proofs of Lemmas
\ref{vertex-count}, \ref{flat-edges} and \ref{vertical-edges} gives an upper bound of
\[
    n\cdot 2^{O(\sqrt{\log n\log\log n})}\cdot (\log n)^{O(t^2)}
\]
on the number of edges and vertices in $U_n$ and establishes \cref{thm:main_induced}.  The bottleneck in the analysis is the value $\lambda$ which represents the tradeoff between the lengths of the transition codes $\mu_y$ and the excess height of trees $T_1,\ldots,T_h$ (this tradeoff is captured by the parameter $k$ in \cite{AdjacencyLabellingPlanarJACM}).  In particular, the optimal tradeoff is obtained when $|\mu_y(v)|\in O(\sqrt{\log n\log\log n})$ and $\h(T_y)\le \log |S^+_y| + O(\sqrt{\log n\log\log n})$.  The $(\log n)^{O(t^2)}$ factor comes from storing the colours $\psi'_{y+b}(p_i(v))$ in each for each $i\in\{1,\ldots,t+1\}$, since each colour comes from a set of size $(\log n)^{O(t)}$.

We remark that our proof includes within it a labelling scheme for graphs of treewidth at most $t$.  Analyzing this labelling scheme separately shows that it gives rise to a graph $H_n$ that has $n(\log n)^{O(t^2)}$ edges and vertices, and contains each $n$-vertex subgraph of treewidth at most $t$ as an induced subgraph.

\section{Conclusion}

Our construction of universal graphs is based on the product structure theorem of~\cite{dujmovic.joret.ea:planarJACM}, which does not apply to every proper minor-closed classes of graphs, only to apex-minor-free classes. A natural problem is thus to construct universal graphs with $o(n^{3/2})$ edges for $n$-vertex graphs from an arbitrary proper minor-closed class.

\section*{Note added in proof}

Very recently, Gawrychowski and Janczewski~\cite{GJ22} showed that the use of bulk tree sequences in~\cite{AdjacencyLabellingPlanarJACM} could be replaced with a simpler approach based on B-trees, while leaving the rest of the proof essentially unchanged.
This simplifies the data-structure part of the proof in~\cite{AdjacencyLabellingPlanarJACM} and also gives a slightly improved bound of $n\cdot 2^{O(\sqrt{\log n})}$ on the number of vertices in the resulting induced-universal graph for $n$-vertex planar graphs, compared to $n\cdot 2^{O(\sqrt{\log n \cdot \log \log n})}$ in~\cite{AdjacencyLabellingPlanarJACM}.
Given our use of the bulk tree sequences from~\cite{AdjacencyLabellingPlanarJACM} as a `black box' in Sections~\ref{sec:prel} and \ref{sec:universal}, they can also be replaced with the approach based on B-trees from~\cite{GJ22} in these proofs.
This reduces the factor $n\cdot 2^{O(\sqrt{\log n \cdot \log \log n})}$ in Theorems~\ref{thm:planar} and \ref{thm:main} to $n\cdot 2^{O(\sqrt{\log n})}$.
On the other hand, the proofs in Section~\ref{sec:induced-universal} do depend on the inner workings of bulk tree sequences, and as such it is not immediately clear whether they could be replaced with B-trees.
As in~\cite{GJ22}, we leave this as an open problem.

\section*{Acknowledgment}

We thank Noga Alon for providing the details of the argument used
in~\cite{ACKRRS} to decrease the number of vertices in a universal
graph.
We are grateful to an anonymous referee for their helpful comments on
an earlier version of the paper. We thank Neel Kaul for spotting a mistake in the original proof of Lemma
\ref{lem:Gn_universal}. This was fixed by using \cite[Lemma
35]{AdjacencyLabellingPlanarJACM} instead of \cite[Lemma
25]{AdjacencyLabellingPlanarJACM} in the statement of Lemma~\ref{lem:tree_sequence}\eqref{bound_nu}.

\bibliographystyle{plainurlnat}
\bibliography{bibliography}

\begin{thebibliography}{19}
\providecommand{\natexlab}[1]{#1}
\providecommand{\url}[1]{\texttt{#1}}
\providecommand{\urlprefix}{URL }
\expandafter\ifx\csname urlstyle\endcsname\relax
  \providecommand{\doi}[1]{\href{https://dx.doi.org/#1}{\nolinkurl{doi:#1}}}\else
  \providecommand{\doi}[1]{\href{https://dx.doi.org/#1}{\nolinkurl{doi:#1}}}\fi
\providecommand{\eprint}[2][]{\url{#2}}

\bibitem[{Alon et~al.(2001)Alon, Capalbo, Kohayakawa, R{\"{o}}dl, Rucinski, and
  Szemer{\'{e}}di}]{ACKRRS}
Noga Alon, Michael~R. Capalbo, Yoshiharu Kohayakawa, Vojtech R{\"{o}}dl,
  Andrzej Rucinski, and Endre Szemer{\'{e}}di.
\newblock Near-optimum universal graphs for graphs with bounded degrees.
\newblock In Michel~X. Goemans, Klaus Jansen, Jos{\'{e}} D.~P. Rolim, and Luca
  Trevisan, editors, \emph{Approximation, Randomization and Combinatorial
  Optimization: Algorithms and Techniques, 4th International Workshop on
  Approximation Algorithms for Combinatorial Optimization Problems, {APPROX}
  2001 and 5th International Workshop on Randomization and Approximation
  Techniques in Computer Science, {RANDOM} 2001 Berkeley, CA, USA, August
  18-20, 2001, Proceedings}, volume 2129 of \emph{Lecture Notes in Computer
  Science}, pages 170--180. Springer, 2001.
\newblock \doi{10.1007/3-540-44666-4\_20}.

\bibitem[{Babai et~al.(1982)Babai, Chung, Erd\H{o}s, Graham, and
  Spencer}]{babai.chung.ea}
Laszlo Babai, Fan R.~K. Chung, Paul Erd\H{o}s, Ronald~L. Graham, and Joel~H.
  Spencer.
\newblock On graphs which contain all sparse graphs.
\newblock In \emph{Theory and practice of combinatorics. A collection of
  articles honoring Anton Kotzig on the occasion of his sixtieth birthday},
  pages 21--26. Elsevier, 1982.

\bibitem[{Bhatt et~al.(1988)Bhatt, Chung, Hong, and
  Rosenberg}]{bhatt1988optimal}
Sandeep Bhatt, Fan R.~K. Chung, Jia-Wei Hong, and Arnold Rosenberg.
\newblock Optimal simulations by butterfly networks.
\newblock In \emph{Proceedings of the twentieth annual ACM symposium on Theory
  of computing}, pages 192--204. 1988.

\bibitem[{Bhatt et~al.(1986)Bhatt, Chung, Leighton, and
  Rosenberg}]{bhatt1986optimal}
Sandeep~N. Bhatt, Fan R.~K. Chung, Frank~T. Leighton, and Arnold~L. Rosenberg.
\newblock Optimal simulations of tree machines (preliminary version).
\newblock In \emph{27th Annual Symposium on Foundations of Computer Science,
  Toronto, Canada, 27-29 October 1986}, pages 274--282. {IEEE} Computer
  Society, 1986.
\newblock \doi{10.1109/SFCS.1986.38}.

\bibitem[{Bhatt et~al.(1989)Bhatt, Chung, Leighton, and
  Rosenberg}]{bhatt.chung.ea}
Sandeep~N. Bhatt, Fan R.~K. Chung, Frank~T. Leighton, and Arnold~L. Rosenberg.
\newblock Universal graphs for bounded-degree trees and planar graphs.
\newblock \emph{SIAM Journal on Discrete Mathematics}, 2(2):145--155, 1989.

\bibitem[{Bose et~al.(2022)Bose, Morin, and
  Odak}]{bose_et_al:LIPIcs.SWAT.2022.19}
Prosenjit Bose, Pat Morin, and Saeed Odak.
\newblock {An Optimal Algorithm for Product Structure in Planar Graphs}.
\newblock In Artur Czumaj and Qin Xin, editors, \emph{18th Scandinavian
  Symposium and Workshops on Algorithm Theory (SWAT 2022)}, volume 227 of
  \emph{Leibniz International Proceedings in Informatics (LIPIcs)}, pages
  19:1--19:14. Schloss Dagstuhl -- Leibniz-Zentrum f{\"u}r Informatik,
  Dagstuhl, Germany, 2022.
\newblock \doi{10.4230/LIPIcs.SWAT.2022.19}.

\bibitem[{Capalbo(2002)}]{Capalbo02}
Michael~R. Capalbo.
\newblock Small universal graphs for bounded-degree planar graphs.
\newblock \emph{Combinatorica}, 22(3):345--359, 2002.
\newblock \doi{10.1007/s004930200017}.

\bibitem[{Chung(1990)}]{chung1990-separator}
Fan R.~K. Chung.
\newblock Separator theorems and their applications.
\newblock In \emph{{Paths, flows, and VLSI-layout, Proc. Meet., Bonn/Ger. 1988,
  Algorithms Comb. 9, 17-34 (1990)}}, pages 17--34. 1990.

\bibitem[{Chung and Graham(1983)}]{chung.graham}
Fan R.~K. Chung and Ronald~L. Graham.
\newblock On universal graphs for spanning trees.
\newblock \emph{Journal of the London Mathematical Society}, s2-27(2):203--211,
  1983.
\newblock \doi{10.1112/jlms/s2-27.2.203}.

\bibitem[{Dujmovi\'c et~al.(2020{\natexlab{a}})Dujmovi\'c, Esperet, Gavoille,
  Joret, Micek, and Morin}]{AdjacencyLabellingPlanarFOCS}
Vida Dujmovi\'c, Louis Esperet, Cyril Gavoille, Gwena\"el Joret, Piotr Micek,
  and Pat Morin.
\newblock Adjacency labelling for planar graphs (and beyond).
\newblock In \emph{61th {IEEE} Annual Symposium on Foundations of Computer
  Science, {FOCS} 2020, Virtual Conference, November 16-19, 2020}.
  2020{\natexlab{a}}.

\bibitem[{Dujmovi\'c et~al.(2021)Dujmovi\'c, Esperet, Gavoille, Joret, Micek,
  and Morin}]{AdjacencyLabellingPlanarJACM}
Vida Dujmovi\'c, Louis Esperet, Cyril Gavoille, Gwena\"el Joret, Piotr Micek,
  and Pat Morin.
\newblock Adjacency labelling for planar graphs (and beyond).
\newblock \emph{Journal of the ACM}, 68(6):Article 42, 2021.
\newblock
  \href{https://arxiv.org/abs/2003.04280}{\nolinkurl{arXiv:2003.04280}}.

\bibitem[{Dujmovi\'c et~al.(2020{\natexlab{b}})Dujmovi\'c, Joret, Micek, Morin,
  Ueckerdt, and Wood}]{dujmovic.joret.ea:planarJACM}
Vida Dujmovi\'c, Gwena\"el Joret, Piotr Micek, Pat Morin, Torsten Ueckerdt, and
  David~R. Wood.
\newblock Planar graphs have bounded queue-number.
\newblock \emph{Journal of the ACM}, 67(4):Article 22, 2020{\natexlab{b}}.
\newblock
  \href{https://arxiv.org/abs/1904.04791}{\nolinkurl{arXiv:1904.04791}}.

\bibitem[{Dujmović et~al.(2023)Dujmović, Morin, and
  Wood}]{dujmovic.morin.ea:structure}
Vida Dujmović, Pat Morin, and David~R. Wood.
\newblock Graph product structure for non-minor-closed classes.
\newblock \emph{Journal of Combinatorial Theory, Series B}, 162:34--67, 2023.
\newblock \doi{10.1016/j.jctb.2023.03.004}.

\bibitem[{Gawrychowski and Janczewski(2022)}]{GJ22}
Pawe{\l} Gawrychowski and Wojciech Janczewski.
\newblock Simpler adjacency labeling for planar graphs with {B}-trees.
\newblock In Karl Bringmann and Timothy Chan, editors, \emph{5th Symposium on
  Simplicity in Algorithms, SOSA@SODA 2022, Virtual Conference, January 10-11,
  2022}, pages 24--36. {SIAM}, 2022.
\newblock \doi{10.1137/1.9781611977066.3}.

\bibitem[{Kannan et~al.(1992)Kannan, Naor, and
  Rudich}]{kannan.naor.ea:implicit}
Sampath Kannan, Moni Naor, and Steven Rudich.
\newblock Implicit representation of graphs.
\newblock \emph{{SIAM} J. Discrete Math.}, 5(4):596--603, 1992.
\newblock \doi{10.1137/0405049}.

\bibitem[{Pilipczuk and Siebertz(2021)}]{pilipczuk.siebertz:polynomial-journal}
Micha{\l} Pilipczuk and Sebastian Siebertz.
\newblock Polynomial bounds for centered colorings on proper minor-closed graph
  classes.
\newblock \emph{Journal of Combinatorial Theory, Series {B}}, 151:111--147,
  2021.
\newblock \doi{10.1016/j.jctb.2021.06.002}.

\bibitem[{Scheffler(1992)}]{scheffler:optimal}
Petra Scheffler.
\newblock Optimal embedding of a tree into an interval graph in linear time.
\newblock In Jaroslav Ne{\v{s}}et{\v{r}}il and Miroslav Fiedler, editors,
  \emph{Fourth Czechoslovakian Symposium on Combinatorics, Graphs and
  Complexity}, volume~51 of \emph{Annals of Discrete Mathematics}, pages
  287--291. Elsevier, 1992.
\newblock \doi{10.1016/S0167-5060(08)70644-7}.

\bibitem[{Ueckerdt et~al.(2022)Ueckerdt, Wood, and Yi}]{UWY22}
Torsten Ueckerdt, David~R. Wood, and Wendy Yi.
\newblock An improved planar graph product structure theorem.
\newblock \emph{Electronic Journal of Combinatorics}, 29(2), 2022.
\newblock \doi{10.37236/10614}.

\bibitem[{Valiant(1981)}]{Val81}
Leslie~G. Valiant.
\newblock Universality considerations in {VLSI} circuits.
\newblock \emph{IEEE Transactions on Computers}, 30(2):135--140, 1981.

\end{thebibliography}

\appendix

\section{Proof of Lemma~\ref{lem:alon}}



Let $N$ be the smallest integer divisible by $k$ such that $N\ge N_0$.
Note that we have $N\le N_0+k$. Consider
an integer $d=\Theta(\tfrac1\epsilon \log k)$ (whose precise value will be
determined later).
Let $H$ be a bipartite graph with parts $V$ of size $N$, and $U$ of
size  $N/k$ in which each vertex of $V$ is connected to $d$ random
vertices of $U$ (with replacement, so it might be the case that some
vertices of $V$ have less than $d$ neighbors).

\begin{claim}\label{cl:hall}
The following holds with positive probability: for each subset $X$ of $V$ with $|X|\le n$, $|N_H(X)|\ge |X|$.
\end{claim}

\begin{proof}
For subsets $S\subseteq V$ and $T\subseteq U$ of size $s\le n$ and $t<s$,
respectively, we denote by $\mathcal{E}_{S,T}$ the event that all neighbors of
$S$ are in $T$. Note that $\mathcal{E}_{S,T}$ occurs with probability
$(tk/N)^{sd}$. By the union bound, the probability that there are two subsets
$S\subseteq V$ of size $s\le n$ and $T\subseteq U$ of size $t<s$, such that all neighbors of $S$ are in $T$ is at most
\begin{eqnarray*}
          \sum_{s=1}^{n}\sum_{t=1}^{s-1}{N\choose s}{N/k \choose
  t}\left(\frac{tk}{N}\right)^{sd} &\le &  \sum_{s=1}^{n}s\left( \frac{Ne}s\right)^s\left( \frac{Ne}{ks}\right)^s\left(\frac{sk}{N}\right)^{sd}\\
  & \le & \sum_{s=1}^{n} s\left[\frac{Ne}{s}\cdot \frac{Ne}{sk}\cdot
          \left(\frac{sk}{N}\right)^{d}\right]^s
          \\
  & \le & \sum_{s=1}^{n} \left[2e^2k \left(\frac{sk}{N}\right)^{d-2}\right]^s,
\end{eqnarray*}
where we have used the inequalities $s\le 2^s$ and ${a \choose b}\le
(ae/b)^b$. We have also used the fact that the function $x\mapsto (c/x)^x$ is
increasing on the interval $(0,c/e)$ for any fixed $c>0$, and thus
${N/k \choose t}\le (\tfrac{Ne}{tk})^t\le (\tfrac{Ne}{sk})^s$ for any
$t\le s\le n\le \tfrac1e(Ne/k)= N/k$.
\medskip

Since $s\le n\le \tfrac{N_0}{(1+\epsilon)k}\le \tfrac{N}{(1+\epsilon)k}$ we
have $(sk/N)^{d-2}\le (\tfrac1{1+\epsilon})^{d-2}\le
\exp(-\tfrac{\epsilon}2(d-2))$ for any $0<\epsilon<1$. It
follows that by taking $d$ to be a sufficiently large in $\Omega(\tfrac1\epsilon \log k)$ we have
$2e^2k\left(\tfrac{sk}{N}\right)^{d-2}<1/10$ and thus the probability that
there exist two subsets
$S\subseteq V$ of size $s\le n$ and $T\subseteq U$ of size  $t<s$, such that all neighbors of $S$ are in $T$ is at most
$\sum_{s\ge 1} 10^{-s}<1$. This completes the proof of the claim.
\end{proof}

The property that any $n$-vertex subset $X$ of $V$ is saturated by a
matching of $H$ is now a direct consequence of Hall's theorem (applied
to the subgraph of $H$ induced by $X\cup U$).
This concludes the proof of Lemma~\ref{lem:alon}.

\medskip

We note that the probabilistic construction of $H$ in the proof above can be replaced by a
purely deterministic (and explicit) construction using expander
graphs, at the cost of a more tedious analysis and a worse bound on
the degree $d$ (as a function of $k$ and $\epsilon$). The advantage of
such an explicit construction is that (together with the other
components of our proof), it provides an explicit description of the
universal graph with $(1+o(1))n$ vertices, and an efficient deterministic algorithm
giving an embedding of any $n$-vertex planar graph in the universal
graph.

\medskip

\begin{figure}[htb]
  \centering
  \pgfornament[width = 3cm, color = black]{72}
\end{figure}

\subsection*{Correspondence}
{\small
\begin{itemize}
  \item Louis Esperet, Laboratoire G-SCOP, 46 avenue F\'elix Viallet, 38000
Grenoble, France.
\item Gwena\"el Joret, Computer Science Department, Universit\'e libre de Bruxelles, Campus de la Plaine, CP 212, 1050 Brussels, Belgium.
\item Pat Morin, School of Computer Science, Carleton University, 1125
  Colonel By Drive, Ottawa, Ontario K1S 5B6, Canada.
\end{itemize}
}
\end{document}